\documentclass[a4paper,10pt]{amsart}
\usepackage[plainpages=false]{hyperref}
\usepackage{amsfonts,latexsym,rawfonts,amsmath,amssymb,amsthm, mathrsfs, lscape}
\usepackage{verbatim}

\usepackage[all]{xy}
\usepackage{graphicx,psfrag}

\usepackage{array, tabularx}

\usepackage{setspace}

\newtheorem{thm}{Theorem}[section]

\newtheorem{lem}[thm]{Lemma}

\newtheorem{prop}[thm]{Proposition}

\newtheorem{conj}[thm]{Conjecture}
\newtheorem{rmk}[thm]{Remark}

\theoremstyle{definition}

\numberwithin{equation}{section}

\def \C {\mathbb C}
\def \Z {\mathbb Z}
\def \R {\mathbb R}

\def \P {\mathbb P}

\def \p {\partial}
\def \bp {\bar{\partial}}
\def \bf {\bar{f}}

\def \Vol {\text{Vol}}

\begin{document}

\title{Gromov-Hausdorff limits of K\"ahler manifolds and algebraic geometry}
\author{Simon Donaldson and Song Sun}
\date{\today}
\maketitle

\newcommand{\oc}{\overline{c}}
\newcommand{\uc}{\underline{c}}
\newcommand{\oY}{\overline{Y}}
\newcommand{\urho}{\underline{\rho}}
\newcommand{\db}{\overline{\partial}}
\newcommand{\bZ}{{\mathbb Z}}
\newcommand{\reg}{{\rm reg}}
\newcommand{\hTr}{\widehat{{\rm Tr}}}
\newcommand{\hDim}{\widehat{{\rm Dim}}}
\newcommand{\hVol}{\widehat{{\rm Vol}}}
\newcommand{\cU}{{\mathcal U}}
\newcommand{\hatX}{\hat{X}}
\newcommand{\osigma}{\overline{\sigma}}
\newcommand{\otau}{\overline{\tau}}
\newcommand{\os}{\overline{s}}
\newcommand{\ord}{{\rm ord}}
\newcommand{\cF}{{\mathcal F}}
\newcommand{\cZ}{{\mathcal Z}}
\newcommand{\hpi}{\hat{\pi}}
\newcommand{\cI}{{\mathcal I}}
\newcommand{\cE}{{\mathcal E}}
\newcommand{\hcX}{\hat{{\mathcal X}}}
\newcommand{\wW}{\widehat{W}}
\newcommand{\tg}{\tilde{g}}
\newcommand{\cA}{{\mathcal A}}
\newcommand{\bQ}{{\bf Q}}
\newcommand{\cX}{{\mathcal X}}
\newcommand{\bC}{\mbox{${\mathbb C}$}}
\newcommand{\bR}{\mbox{${\mathbb R}$}}
\newcommand{\Var}{{\rm Var}}
\newcommand{\Av}{{\rm Av}}

\newcommand{\Dim}{{\rm Dim}}
\newcommand{\cO}{{\mathcal O}}
\newcommand{\cW}{{\mathcal W}}
\newcommand{\cL}{{\mathcal L}}
\newcommand{\Tr}{{\rm Tr}}
\newcommand{\Zmax}{Z_{{\rm max}}}
\newcommand{\Zmin}{Z_{{\rm min}}}
\newcommand{\Ch}{{\rm Ch}}
\newcommand{\bP}{\mbox{${\mathbb P}$}}
\newcommand{\uA}{\mbox{${\underline{A}}$}}
\newcommand{\uM}{\mbox{${\underline{M}}$}}
\newcommand{\um}{\mbox{${\underline{m}}$}}
\newcommand{\ur}{\mbox{${\underline{r}}$}}
\newtheorem{Goal}{Goal}
\newtheorem{question}{Question}

\newtheorem{defn}{Definition}
\newcommand{\oK}{\overline{K}}
\newcommand{\dbd}{\sqrt{-1} \partial\overline{\partial}}
\newcommand{\ulambda}{\underline{\lambda}}
\newcommand{\olambda}{\overline{\lambda}}
\newcommand{\Riem}{{\rm Riem}}
\newcommand{\Ric}{{\rm Ric}}

\section{Introduction}

The main purpose of  this paper is to prove a general result about the geometry of holomorphic line bundles over K\"ahler manifolds.
This result is essentially a verification of a conjecture of Tian \cite{Tian2} and Tian has, over many years, highlighted the importance of the question for the existence theory of K\"ahler-Einstein metrics. We will begin by stating this main result. 

We consider data $(X,g,J,L,A)$  where $(X,g)$ is a compact Riemannian manifold of real dimension $2n$, $J$ is a complex structure which respect to which the metric is K\"ahler, $L$ is a Hermitian line bundle over $X$ and $A$ is a connection on $L$ with curvature $-i \omega$, where $\omega$ is the K\"ahler form.  We will often just write $X$ as an abbreviation for this data.  We suppose the metric satisfies fixed upper and lower bounds on the Ricci tensor
\begin{equation}     -\frac{g}{2}\leq \Ric\leq g. \end{equation}
(The  particular bounds we have chosen are just convenient normalisations; any other fixed bounds would do.) For $V,c>0$  let ${\mathcal K}(n,c,V)$ denote the class of all such data such that the volume of $X$ is $V$ and the \lq\lq non-collapsing'' condition
\begin{equation}    {\rm Vol}\ B_{r} \geq c \frac{\pi^{n}}{n!}r^{2n} \end{equation}
holds. Here $B_{r}$ is any metric $r$-ball in $X$, $r$ is any number less than the diameter of $X$ and the normalising factor $\pi^{n}/n!$ is the volume of the unit ball in $\bC^{n}$. 

\

The connection induces a holomorphic structure on $L$ and for each  positive integer $k$ there is a natural $L^{2}$ hermitian metric on the space $H^{0}(X,L^{k})$. Recall that the \lq\lq density of states'' (or Bergman) function $\rho_{k,X}$ is defined by
$$  \rho_{k,X}= \sum \vert s_{\alpha}\vert^{2}, $$
where $(s_{\alpha})$ is any orthonormal basis of $H^{0}(X,L^{k})$.  An equivalent definition is that $\rho_{k,X}(x)$ is the maximum of $\vert s(x)\vert^{2}$ as $s$ runs over the holomorphic sections with $L^{2}$ norm $1$. Thus, to establish a lower bound on $\rho_{k,X}(x)$ we have to produce a holomorphic section $s$ with $L^{2}$ norm not too large and with $\vert s(x)\vert$ not too small. Write $$\urho(k,X)=\min_{x\in X} \rho_{k,X}(x).$$
Standard theory, a part of the Kodaira Embedding Theorem, asserts that
for each fixed $X$ we have $\urho(k,X)>0$ for large enough $k$. Our main result can be thought of as an extension of this statement which is both {\it uniform} over ${\mathcal K}(n,c,V)$ and gives a   definite lower bound.

\begin{thm}
Given $n,V, C$ there is an integer $k_{0}$ and $b>0$ such that $\urho(k_{0},X)\geq b^{2}$ for all $X$ in ${\mathcal K}(n,c,V)$.
\end{thm}

The proof involves a combination of the Gromov-Hausdorff convergence theory--- developed by Anderson, Cheeger, Colding, Gromoll,  Gromov, Tian and others over the past 30 years or so---and the \lq\lq H\"ormander technique'' for constructing holomorphic sections. When $n=2$ the theorem was essentially proved by Tian in \cite{Tian1} and the overall scheme of our proof is similar. The theorem provides the foundations for a bridge between the differential geometric convergence theory and algebraic geometry, leading to the following result (as indicated by Tian). 
\begin{thm}
 Given $n,c,V$ there is a fixed $k_{1}$ and integer $N$ with the following effect. 
 \begin{itemize}
 \item Any $X$ in ${\mathcal K}(n,c,V)$ can be embedded in a linear subspace of $\bC\bP^{N}$ by  sections of $L^{k_{1}}$.
\item 
Let $X_{i}$ be a sequence in ${\mathcal K}(n,c,V)$ with Gromov-Hausdorff limit $X_{\infty}$. Then $X_{\infty}$  is homeomorphic to a normal projective variety $W$ in $\bC\bP^{N}$. After passing to a subsequence and taking a suitable sequence of projective transformations we can suppose that the projective varieties $X_{i} \subset \bC\bP^{N}$ converge as algebraic varieties to $W$. 
\end{itemize}
\end{thm}

(More precise statements, and more detailed information,  are given in Section 4 below.)

Many of the ideas and arguments required to derive this  are similar to those of Ding and Tian in \cite{DT} who considered  Fano manifolds with K\"ahler-Einstein metrics.  Then  the limit is a \lq\lq $\mathbb Q$-Fano'' variety, as Ding and Tian  conjectured. 

\

In Section 2 we review relevant background in convergence theory and complex differential geometry. Given this background, the rest of the proof is essentially self-contained. This proof is given in Section 3. We begin by reducing Theorem 1.1 to a \lq\lq local'' statement (Theorem 3.2)involving a point in a Gromov-Hausdorff limit space and attention is then focused on a tangent cone at this point. In Section 4 we give the proof of Theorem 1.2. We also establish some further relations between the differential geometric and algebro-geometric theories. In Section 5 we include a more detailed analysis of  tangent cones in the 3-dimensional case, showing that these are cones over Sasaki-Einstein orbifolds and discuss the likely picture in the higher dimensional situation.

Our main interest throughout this paper is in the case when $X$ is a Fano manifold, the metric is K\"ahler-Einstein with positive Ricci curvature and $L=K_{X}^{-1}$. Then the Ricci bound (1.1) holds trivially and the non-collapsing condition (1.2) is automatic for a suitable $c$ (in fact with $c=V (2n-1)!!/(2^{n+1}(2n-1)^{n} \pi^{n})$). But the general hypotheses we have made above   seem to give the natural context  for the discussion here, although the applications outside the Fano case may be limited.  In the case of K\"ahler-Einstein metrics of  negative or zero Ricci curvature there is of course a complete existence theory due to Aubin and Yau. It would be interesting to characterise the non-collapsing condition in this situation algebro-geometrically.

\

We would like to emphasise  that this is a \lq\lq theoretical'' paper in the following sense. Our purpose is to establish that some high powers $k_{0}, k_{1}$ of a positive line bundle have certain good properties uniformly over manifolds in  ${\mathcal K}(n,c,V)$ and Gromov-Hausdorff limits thereof. For many reasons one would like to know values of  $k_{0}, k_{1}$ which are, first, explicitly computable and,  second, realistic. (That is, not too different from the optimal values which, in reality, yield these good properties.)  This paper is theoretical in that we will not attempt to  do anything of this kind. Of  our two foundations---the H\"ormander technique and convergence theory----the first is quite amenable to explicit estimates but the second is not. So it is unclear whether even in principle one could extract any  computable numbers. Of course this is an important question for future research. Given this situation we have not attempted to make the arguments in Section 3 and 4 efficient, in the sense that (even if one somehow had effective constructions for the building blocks) our arguments from those building blocks who lead to huge, completely unrealistic, numbers. This is connected to certain definite and tractable mathematical questions which we take up briefly again at the very end of the paper, where we formulate a conjectural sharper version of Theorem 1.1 (Conjecture 5.15).

\

We finish this introduction with some words about the origins  of this paper. While the question that we answer in Theorem 1.1 is a central one in the field of K\"ahler-Einstein geometry, it is not something that the authors have focused on until recently. The main construction in this paper emerged as an off-shoot of a joint project by the  first-named  author and Xiuxiong Chen, studying the slightly different problem of K\"ahler-Einstein metrics with cone singularities along a divisor. A companion article by the first named author and Chen, developing this related theory, will appear shortly.  Both authors are very grateful to Chen for discussions of these matters, extending over many years. 
   
\section{Background}
\subsection{Convergence Theory}
This subsection is a rapid summary of many formidable results. We will not attempt to give detailed references, but refer to the surveys \cite{Ch1} \cite{Ch2}and the references therein.

Recall that if $Z,W$ are two compact metric spaces then the Gromov-Hausdorff distance $d_{GH}(Z,W)$ is the infimum of numbers $\delta$ such that there is a metric on $Z\sqcup W$ extending the given metrics on the components and such that each of $Z,W$ is $\delta$-dense. The starting point of the theory is Gromov's Theorem that a sequence of compact $m$-dimensional Riemannian manifolds $(M_{i}, g_{i}) $ with bounded diameter and with Ricci curvature bounded below has a Gromov-Hausforff convergent subsequence  with some limit $M_{\infty}$, which is a compact metric space. In our situation, with a sequence in ${\mathcal K}(n,c,V)$ the diameter bound follows from the non-collapsing condition (1.2). Passing to this subsequence, we can fix metrics on the disjoint unions $M_{i}\sqcup M_{\infty}$ such that $M_{i}, M_{\infty}$ are $\delta_{i}$ dense, where $\delta_{i}\rightarrow 0$. 
\

Now suppose that,  as in our situation, the Ricci tensors  of $M_{i}$ satisfy fixed upper and lower bounds. Suppose also that a non-collapsing condition (1.2) holds and the volumes are bounded below. Then there is a connected, open, dense subset $M^{\reg}_{\infty}\subset M_{\infty}$ which is an $m$-dimensional $C^{2,\alpha}$-manifold and has a $C^{1,\alpha}$ Riemannian metric $g_{\infty}$ (for all H\"older exponents $\alpha<1$.) This is compatible with the metric space structure in the following sense.  For any compact $K\subset M_{\reg}$ we can find a number $s>0$ such that if $x_{1}, x_{2}$ are points of $K$ with $d(x_{1}, x_{2})\leq s$ then $d(x_{1}, x_{2})$ is the infimum of the length of paths in $M_{\reg}$ between $x_{1}, x_{2}$.  Moreover the convergence on this subset is $C^{1,\alpha}$, in the following sense. Given any number $\delta>0$ and  compact subset $K\subset M_{\reg}$  we can find for large enough $i$ open embeddings $\chi_{i}$ of an open neighbourhood of $K$ into $M_{i}$ such that:
\begin{enumerate} \item The pull-backs by the $\chi_{i}$ of the $g_{i}$ converge in $C^{1,\alpha}$ over $K$ to $g_{\infty}$. 
\item   $d(x, \chi_{i}(x))\leq \delta$ for all $x\in K$.
\end{enumerate}
The second item here refers to the chosen metric on $M_{i}\sqcup M_{\infty}$.
\

\

The volume form of the Riemannian metric on the dense set $M_{\infty}^{\reg}$ defines a measure on $M_{\infty}$ and the volume of $M_{\infty}$ is the limit of the volumes of the $M_{i}$. The Hausdorff dimension of the {\it singular set} $\Sigma=M_{\infty}\setminus M_{\infty}^{\reg}$ does not exceed $m-2$.

There is a variant of the theory in which one considers spaces with base points and convergence over bounded distance from the base points. In particular we can take a point $p\in M_{\infty}$ and any sequence $R_{i}\rightarrow \infty$ and then consider the sequence of based metric spaces given by scaling $M_{\infty}$ by a factor $R_{i}$. The compactness theorem implies that, passing to a subsequence, we get convergence  and a fundamental result is that, under our hypotheses,  the limit of such is a metric cone $C(Y)$---a {\it tangent cone} of $M_{\infty}$ at $p$. Here $Y$ is a metric space which contains a dense open subset $Y^{\reg}$ which is a {\it smooth} $(m-1)$-dimensional Einstein manifold, with Ricci curvature equal to $(m-2)$, and the metric and Riemannian structures on $Y^{\reg}$ are related in a similar way to that above. Likewise for the natural measure on $Y$. The singular set $\Sigma_{Y}=Y\setminus Y^{\reg}$ has Hausdorff dimension at most $m-3$. The cone $C(Y)$ has a  smooth Ricci-flat metric outside the singular set $\{O\}\cup C(\Sigma_{Y})$ (where $O$ is the vertex of the cone) and the convergence of the rescaled metrics is $C^{1,\alpha}$ in the same sense as before.  An important numerical invariant of this situation is the volume ratio
\begin{equation}     \kappa = \frac{\Vol(Y)}{\Vol(S^{m-1})}. \end{equation}
The Bishop inequality implies that  $\kappa\leq 1$ and if a noncollapsing bound like (1.2) holds for the original manifolds $M_{i}$ we have $\kappa\geq c$. 

 All of the preceding discussion is in the general Riemannian context. Suppose now that our manifolds are $X_{i}, g_{i}$ and we have additional structures $J_{i}, L_{i}, A_{i}$ as in Section 1. We define a {\it polarised limit space} to be a metric limit
$X_{\infty}, g_{\infty}$ as above, together with extra data as follows
\begin{itemize}
\item  A $C^{1,\alpha}$ complex structure $J_{\infty}$ on the regular set with respect to which the metric is K\"ahler with $2$-form $\omega_{\infty}$.
\item A $C^{2,\alpha}$ line bundle $L_{\infty}$ over the regular set;
\item A $C^{1,\alpha}$ connection $A_{\infty}$ on $L_{\infty}$ with curvature $-i\omega_{\infty}$.
\end{itemize}
(Notice that the integrability theorem for complex structures extends  to the $C^{1,\alpha}$ situation \cite{NW}. Thus in fact we could say that $X^{\reg}_{\infty}$ and $L$ have smooth structures while the convergence is in $C^{1,\alpha}$. But this is largely irrelevant for our purposes.)

\

We define convergence of a sequence $(X_{i}, J_{i}, L_{i},A_{i})$ to such a polarised limit by requiring that for compact $K\subset X_{\infty}^{reg}$ we have maps $\chi_{i}$ as before but in addition so that the pulled back complex structures $\chi_{i}^{*}(J_{i})$ converge to $J_{\infty}$ (in $C^{1,\alpha}$) and we have bundle isomorphisms $\hat{\chi}_{i}: L_{\infty}\rightarrow \chi^{*}_{i}(L_{i})$
with respect to which the connections converge to $A_{\infty}$ in $C^{1,\alpha}$. It is straightforward to extend the compactness theorem to this polarised situation, using the fact that  $J_{i}$ is a covariant-constant tensor with respect to the $C^{,\alpha}$ Levi-Civita connection of $g_{i}$.

 Likewise, the regular part of a tangent cone $C(Y)$ at a point in $X_{\infty}$ has a smooth, Ricci-flat, K\"ahler metric which is induced  from a {\it Sasaki-Einstein} structure on $Y^{\reg}$. In particular the K\"ahler form on the smooth part can be written as $\frac{i}{2}\partial\db \vert z\vert^{2}$, where $\vert z\vert$ denotes the distance to the vertex of the cone.  

A significant difference in the K\"ahler case is that the singular sets (both in $X_{\infty}$ and in $Y$) are known to have Hausdorff codimension at least $4$. (This is conjectured but not established in the real case.) In particular, which will be crucial for us, the codimension is strictly greater than  $2$.

In the case of primary interest---K\"ahler-Einstein metrics--we obtain  $C^{\infty}$ convergence on compact subsets of the regular sets. (This is also part of the standard literature.) In addition, if we consider the \lq\lq Fano case'', when $L_{i}$ is the anticanonical bundle of $X_{i}$, then the limit line bundle is just the anticanonical bundle of the regular set in $X_{\infty}$. The reader may well prefer to restrict to this case. More generally, if we just assume (in addition to (1.1), (1.2)) that the metrics have constant scalar curvature, then one can still establish this $C^{\infty}$ convergence, for example using the results of Chen and Weber \cite{CW}.

\subsection{Complex differential geometry: the Horm\"ander technique}

We begin by recalling that, under our hypotheses, there is a uniform Sobolev inequality 
\begin{equation}  \Vert f\Vert_{L^{2n/(n-1)}}\leq C_{1} \Vert \nabla f \Vert_{L^{2}}+ C_{2} \Vert f \Vert_{L^{2}}, \end{equation}
for functions $f$ on a manifold $X$ in the class ${\mathcal K}(n,c,V)$, where  $C_{1}, C_{2}$ depend only on $n,c,V$ \cite{Croke}. Here of course we are referring to  norms defined by the metric $g$. When working with the line bundle $L^{k}$ it will be convenient to use the norms defined by the rescaled metrics $kg$ (for integers $k\geq 1$). Thus lengths are scaled by $\sqrt{k}$ and  volumes by $k^{n}$. We will use the notation $L^{2,\sharp}$ etc. to denote norms defined by these rescaled metrics. Then the scaling weight gives
\begin{equation}   \Vert f\Vert_{L^{2n/(n-1),\sharp}}\leq C_{1} \Vert \nabla f \Vert_{L^{2,\sharp}}+ C_{2} k^{-1/2} \Vert f \Vert_{L^{2,\sharp}}. \end{equation}
So the scaling only helps in the Sobolev inequality. Of course the Ricci tensor $\Ric^{\sharp}$ of the rescaled metric is bounded between $-1/(2k)$ and $1/k$.   

\begin{prop}\begin{enumerate}
\item There are constants $K_{0}, K_{1}$, depending only on $n,c,V$ such that if $X$ is in ${\mathcal K}(n,c,v)$ and $s$ is a holomorphic section of $L^{k}$ (for any $k>0$) we have
$$ \Vert s \Vert_{L^{\infty,\sharp}} \leq K_{0} \Vert s \Vert_{L^{2,\sharp}}\ \ , \Vert \nabla s \Vert_{L^{\infty,\sharp}} \leq K_{1} \Vert s \Vert_{L^{2,\sharp}}. . $$
\item If $X$ is in ${\mathcal K}(n,c,V)$ then for any $k>0$ the Laplacian $\Delta_{\db}$ on $\Omega^{0,1}(L^{k})$ is invertible and $\Delta_{\db}^{-1} \leq 2$. 
\end{enumerate}
\end{prop}
In the second item $\Delta_{\db}= \db^{*}\db+ \db \db^{*}$, with adjoints defined using the rescaled metric, and the statement is that for all $\phi$
\begin{equation}  \langle \Delta_{\db}^{-1} \phi, \phi\rangle_{\sharp} \leq 2  \Vert \phi \Vert^{2}_{L^{2,\sharp}}\end{equation}

This Proposition summarises results which are  well-known to workers in the field and which all hinge on various formulae of Bochner-Weitzenbock type. We use the rescaled metrics throughout the discussion. First on  $C^{\infty}$ sections $s$ of $L^{k}$ we have
$$   \nabla^{*}\nabla s = 2 \db^{*} \db s+ s, $$
so when $s$ is holomorphic $\nabla^{*}\nabla s = s$ which implies that \begin{equation}\Delta \vert s\vert \leq \vert s\vert, \end{equation}  where the lack of differentiability of $\vert s\vert$ at the zero set is handled in a standard way. (Note that we  use the \lq\lq geometers convention'' for the sign of the Laplacian in this paper.) Now the bound on the $L^{\infty}$ norm follows from the Moser iteration argument applied to this differential inequality, using the uniform Sobolev inequality (see \cite{Tian2}).

The first derivative bound is obtained in a similar way. Changing notation slightly, for a holomorphic section $s$ with $\db s=0$ we write $\nabla s = \partial s$ where $$\partial:\Omega^{p,q}(L^{k})\rightarrow \Omega^{p+1,q}(L^{k})$$ is defined using the connection. Since $\partial^{2}=0$ we have
$$  \Delta_{\partial } \partial s = \partial \Delta_{\partial} s, $$
where $\Delta_{\partial}=\partial^{*}\partial + \partial \partial^{*}$. Then for a holomorphic section $s$ , $\Delta_{\partial} s =\nabla^{*}\nabla s = s$ and
$$   \Delta_{\partial} (\partial s) = \partial s. $$
Now the  Bochner-Weitzenbock formula comparing $\Delta_{\partial}$ and $\nabla^{*}\nabla$ on $\Omega^{1,0}(L^{k})$ has the form
$$   \Delta_{\partial} = \nabla^{*}\nabla -1+\Ric^{\sharp}, $$
so $$(\nabla^{*}\nabla (\partial s), \partial s)\leq \frac{5}{2} \vert\partial s\vert^{2}. $$
It follows that $$\Delta \vert \partial s \vert\leq \frac{5}{2} \vert \partial s \vert, $$
and the Moser argument applies as before. Notice that, with some labour, the constants $K_{0}, K_{1}$ could be computed explicitly in terms of $n,c,V$. 
 
 For the second item in the Proposition we need a Bochner-Weizenbock formula on $\Omega^{0,1}(L^{k})$ i.e. sections of the bundle $\overline{T}^{*}\otimes L^{k}$. We decompose the covariant derivative on this bundle into (0,1) and (1,0) parts: $\nabla = \nabla'+\nabla''$. Then the formula we want is
\begin{equation}   \Delta_{\db}= (\nabla'')^{*} \nabla'' + \Ric^{\sharp}+ 1 \end{equation}
Given this we have, in the operator sense, $\Delta_{\db}\geq 1/2$ since $\Ric^{\sharp}\geq -1/2$ from which the invertibility and bound on the inverse follow immediately. An efficient way to derive (2.6) is to make the identification
$$\Omega^{0,1}(L^{k})= \Omega^{n,1} (K^{*}\otimes L^{k}), $$ under which $\nabla''$ becomes identified with
$$ \partial^{*} :\Omega^{n,1} (K_{X}^{*}\otimes L^{k})\rightarrow \Omega^{n-1,1}(K_{X}^{*}\otimes L^{k}). $$
The formula (2.6) then becomes a special case of the Kodaira-Nakano formula (\cite{GH} p.154), using the fact that the Ricci form is the curvature of $K_{X}^{*}$.

\

With this background in place we move on to recall a version of the \lq \lq H\"ormander'' construction of holomorphic sections. Suppose we have the following data
\begin{itemize}\item A (non-compact)  manifold $U$, a base point $u_{*}\in U$ and an open neighbourhood $D\subset\subset U$ of $u_{0}$.
\item A $C^{\infty}$ Hermitian line bundle $\Lambda\rightarrow U$.
\item A complex structure $J$ and K\"ahler metric $g$ on $U$ with K\"ahler form $\Omega$.
\item A  connection $A$ on $\Lambda$ having curvature  $-i\Omega$.
\end{itemize}
We use this connection to define a $\db$-operator on sections of $\Lambda$, and hence  a holomorphic structure.

We define a \lq \lq Property (H)'' which this data might have. Fix any $p>2n$.\\

 Property (H):\\

{\em There is a number $C>0$ and a compactly supported section $\sigma$ of $\Lambda\rightarrow U$ such that the following hold.

{ H1:} $ \Vert \sigma\Vert_{L^{2}}<  (2\pi)^{n/2};$

{ H2:} $\vert \sigma(u_{*})\vert > 3/4;$

 { H3:} For any smooth section $\tau$ of $\Lambda$ over a neighbourhood of $\overline{D}$ we have
$$   \vert \tau(u_{*})\vert \leq C \left( \Vert \db \tau \Vert_{L^{p}(D)} +  \Vert \tau \Vert_{L^{2}(D)}\right);$$

{ H4:}  $\Vert \db \sigma\Vert_{L^{2}}< \min( 1/(8\sqrt{2} C), (2\pi)^{n/2}/10 \sqrt{2});$

{ H5:} $\Vert \db \sigma\Vert_{L^{p}(D)}< 1/(8C).$}

\

Many of the specific numbers here are arbitrary but it is convenient to fix some definite numbers. 

We have
\begin{lem} Property (H) is open with respect to variations in $(g,J,A)$ (for fixed $(U,D,u_{*}, \Lambda)$) and the topology of convergence in $C^{0}$ on compact subsets of $U$.
\end{lem}

Notice first that for any choice of data  there is {\it some} constant $C$ for which  the bound in (H3) holds. This follows from the elliptic estimate
    \begin{equation}    \Vert \tau \Vert_{L^{p}_{1}(D_{0})} \leq C_{3} \left( \Vert \db \tau \Vert_{L^{p}(D)} + \Vert \tau \Vert_{L^{2}(D_{0})} \right), \end{equation}
and the Sobolev inequality
$$   \vert \tau(u_{*})\vert \leq C_{4} \Vert \tau\Vert_{L^{p}_{1}(D_{0})}. $$
Here $D_{0}\subset D$ is some interior domain containing $u_{*}$. 
We can write the $\db$-operator on functions for a perturbed complex structure as $\db+\mu \partial$ where $\mu$ is a \lq\lq Beltrami differential''. Similarly if the variation of the connection is given by a $1$-form $a$ then the perturbed $\db$-operator on sections can be written as
      $$(\db+ \mu \partial) + (a'' + \mu a'), $$
 where $a=a'+a''$ is the decomposition into type. It follows that if $\mu$ and $a$ are small in $C^{0}$ then the perturbation of the $\db$ operator is small  in the $L^{p}_{1}\rightarrow L^{p}$ operator norm and it is then clear that the inequality in the third item holds for the perturbed operator, with a slightly larger  constant $C$. For the perturbed structure we use the same section $\sigma$, so the first and second item is automatic. Then it is also clear that, for sufficiently small perturbations, the bounds in the fourth and fifth item (with a slightly larger constant $C$) are also preserved, since we impose strict inequality.

For a connection $A$ on a line bundle $L$ write $A^{\otimes k}$ for the induced connection on $L^{k}$. The following proposition---basically well-known---will provide the core of our proof of Theorem 1.
\begin{prop} Suppose $(X,g,J, L, A)$ is in ${\mathcal K}(n,c,V)$ and $(U,D,u_{*})$ are as above. Suppose that  $\chi:U\rightarrow X$ is an open embedding and the data $$\chi^{*}(J), \chi^{*}(kg), \chi^{*}(L^{k}), \chi^{*}(A^{\otimes k})$$ has Property (H).
Then there is a holomorphic section $s$ of $L^{k}\rightarrow X$ with $L^{2,\sharp}$ norm at most $(11/10)(2\pi)^{n/2}$ and with $\vert s(x)\vert\geq 1/4$ at all points $x$ a distance (in the scaled metric) less than $(4K_{1})^{-1}$ from $\chi(u_{*})$.
\end{prop}

\

 To prove this we transport the section $\sigma$ using the maps $\chi,\hat{\chi}$ and regard it as a smooth section of $L^{k}$ over $X$, extending by zero. The norms we considered over $U$ match up with the $\sharp$-norms over $X$. We write $s= \sigma- \tau$ where $\tau= \db^{*} \Delta_{\db}^{-1} \db \sigma$. By simple Hodge Theory we have $\db s=0$. Now
$$\Vert \tau\Vert_{L^{2,\sharp}} = \langle \Delta_{\db}^{-1} \db \sigma, \db\ \db^{*}\Delta_{\db}^{-1}\db \sigma \rangle = \langle \Delta_{\db}^{-1} \db \sigma, \db \sigma \rangle, $$
since $\db\ \db \sigma=0$. Thus
\begin{equation} \Vert \tau\Vert_{L^{2,\sharp}}\leq \sqrt{2} \Vert \db \sigma \Vert_{L^{2,\sharp}}\leq \min( (8C)^{-1}, \frac{1}{10}(2\pi)^{n/2}). \end{equation}
Hence in particular $$\Vert s\Vert_{L^{2,\sharp}}\leq \Vert \sigma\Vert_{L^{2,\sharp}}+\Vert \tau\Vert_{L^{2,\sharp}}\leq \frac{11}{10} 2\pi^{n/2}. $$
Now work over the image  $\chi(D)$. Applying item (H3) to the section $\tau$ and using (H4), (H5)  we get $\vert \tau(\chi(u_{*}))\vert \leq 1/4$, so $\vert s(\chi(u_{*}))\vert \geq 1/2$. By the derivative bound, $\vert s\vert$ exceeds $1/4$ at points a distance less than $(4K_{1})^{-1}$ from $\chi(u_{*})$.

 To sum up we have the following.

\begin{prop}
Suppose that $U,D,u_{*}, \Lambda$ are as above and data $g_{0},J_{0}, A_{0}$ has Property (H). Then there is some $\psi>0$ with the following effect. Suppose that $X,g_{X}, J_{X}, L, A_{X})$ is in $ {\mathcal K}(n,c, V)$. If we can find $k>0$,   an open embedding  $\chi:U\rightarrow X$ and a bundle isomorphism
$\hat{\chi}: \Lambda\rightarrow \chi^{*}(L^{k})$  such that 
$$\Vert \chi^{*}(J)-J_{0}\Vert_{U}, \Vert \chi^{*}(kg)-g_{0}\Vert_{U}, \Vert \chi^{*}(A^{\otimes k}) - A \Vert_{U} \leq \psi, $$
 then there is a holomorphic  section $s$ of $L^{k}\rightarrow X$ with $L^{2,\sharp}$ norm at most $(11/10)(2\pi)^{n}$ and with $\vert s(x)\vert\geq 1/4$ at all points $x$ a distance (in the scaled metric) less than $(4K_{1})^{-1}$ from $\chi(u_{*})$

\end{prop}

This is just a direct combination  of Lemma 2.2 and Proposition 2.3.  (Here we use the notation $\Vert \ \Vert_{U}$ to indicate the $C^{0}$-norm over $U$.)

\

To illustrate this, take the case when $U$ is the ball of radius $R>2$ in $\bC^{n}$ with the standard flat metric and standard K\"ahler form $\Omega_{0}$. Let $\Lambda$ be the trivial holomorphic line bundle with metric $\exp(-\vert z\vert^{2}/2)$ so the trivialising section, $\sigma_{0}$ say,  has norm $\exp(-\vert z\vert^{2}/4)$ and the induced connection $A_{0}$ has curvature $-i\Omega_{0}$ as required. Let $u_{*}$ be the origin and $D$ be the unit ball. Let $\beta_{R}$ be a standard cut-off function of $\vert z\vert$, equal to $1$ when $\vert z\vert \leq R/2$ and vanishing when $\vert z\vert \geq 9R/10$. Define $\sigma=\beta_{R}\sigma_{0}$. Then we have $\db \sigma= (\db \beta_{0}) \sigma_{0}$. The $L^{2}$ norm of $\sigma$ is slightly less than $(2\pi)^{n}$ and $\vert \sigma(0)\vert =1$. The section $\sigma$ is holomorphic over $D$, so we get (H4) and there certainly is some constant $C$ as in item (H3) of Property (H), independent of $R$. It is clear that, because of the exponential decay, we can fix $R$ so that item (H5) is satisfied. So we have a set of data satisfying Property (H). Now let $x$ be a point in some X in ${\mathcal K}(n,c,V)$. Since the ball is simply connected $U(1)$ connections over it are determined up to isomorphism by their curvature tensors. It is then clear that, when $k$ is sufficiently large, we can find a map $\chi$ with $\chi(0)= x$ and such that the pull back of $kg_{X}, J_{X},  A_{X}^{\otimes k}$ differs by an arbitrarily small amount from the model $g_{0}, J_{0}, A_{0}$. Then we construct a holomorphic section of $L^{k}\rightarrow X$, of controlled $L^{2}$ norm and of a definite positive size on a definite neighbourhood of $x$.  

\begin{rmk}
  There are many possible variants of our ÒProperty HÓ which will end up having the same effect. In particular one can avoid the $L^{p}$ theory. In the context we work in,  we have a first derivative bound as in Prop. 2.1 (1),  and it is easy to show using this that the $L^{2}$ norm of $\tau$ controls $\vert \tau(\chi(u_{*}))\vert$.
\end{rmk}

\section{Proof of Theorem 1.1}
\subsection{Reduction to the local case}
We begin with a simple observation.
       \begin{lem}
             For any integer $\mu\geq 1$ and any $k$ we have
               $$  \rho_{\mu k,X}(x)\geq (K_{0}^{2} k^{n})^{1-\mu} \rho_{k,X}(x)^{\mu}, $$ where $K_{0}$ is the constant in the $C^{0}$-bound of Proposition 2.1. 
               \end{lem}
               Transforming the $C^{0}$-bound to the {\it unscaled} norms gives, for any holomorphic section of $L^{k}$:
               $$  \Vert s\Vert_{L^{\infty}}\leq K_{0} k^{n/2} \Vert s\Vert_{L^{2}}. $$
Write $\rho=\rho_{k,X}(x)$ so there is a section $s$ with $L^{2}$ norm $1$ and with $\vert s(x)\vert^{2}=\rho$. Then $s^{\mu}$ is a holomorphic section of $L^{k\mu}$ with
$$   \vert s^{\mu}(x)\vert^{2}=\rho^{\mu} \ \ \ \Vert s^{\mu}\Vert^{2}_{L^{2}}\leq \Vert s\Vert^{2\mu-2}_{L^{\infty}}\Vert s\Vert^{2}_{L^{2}}\leq K^{2\mu-2} k^{n(\mu-1)}, $$
from which the result follows.

We will use this several times below. In the context of our remarks in the Introduction, note that when $\mu$ is large this gives a rather poor estimate compared with what one would hope to be true, but it suffices for our purposes. 

\begin{thm}
 Let $p$ be a point in a  space $X_{\infty}$ which is a Gromov-Hausdorff limit of manifolds in ${\mathcal K}(n,c,V)$.  There are real numbers $b(p), r(p)>0$ and an integer $k(p)$  with the following effect. Suppose $X_{i}$ in ${\mathcal K}(n,C,V)$ has Gromov-Hausdorff limit $X_{\infty}$. Then there is some $k\leq k(p)$ such that for sufficiently large $i$,  if $x$ is a point in $X_{i}$ with $d(x,p)\leq r(p)$ then $\rho_{k, X}(x)\geq b(p)$.
\end{thm} 

Here, as before, we assume we have fixed metrics on the $X_{i}\sqcup X_{\infty}$.
\begin{prop}
Theorem 3.2 implies Theorem 1.1.
\end{prop}

{\em Proof of Proposition 3.3}

\begin{lem}
Let $X_{\infty}$ be a limit space then, assuming the truth of Theorem $3.2$, there is an integer $k_{X_{\infty}}$ and a $b_{X_{\infty}}>0$ such that if $X_{i}\in {\mathcal K}(n,C,V)$ has Gromov-Hausdorff limit $X_{\infty}$ then for sufficiently large $i$ we have $\urho(k_{X_{\infty}}, X_{i})\geq b^{2}_{X_{\infty}}$.
\end{lem}

We first use the compactness of $X_{\infty}$. The $r(p)/2$-balls centred at points $p$ cover $X_{\infty}$ so we can find a finite sub-cover by balls of radius $r(p_{\alpha})/2$ centred at points $p_{\alpha}\in X_{\infty}$. Let $r$ be the minimum of the $r(p_{\alpha})$. 
Let $i$ be large enough that for any $x\in X_{i}$ there is a point $x_{\infty}\in X_{\infty}$ with $d(x, x_{\infty})\leq r/4$. In addition suppose that $i\geq {\rm max}_{\alpha}i(p_{\alpha})$. Then $x_{\infty}$ lies in the $r(p_{\alpha})/2$ ball centred at $p_{\alpha}$ for some $\alpha$ and hence
$d(x, p_{\alpha})< \frac{3}{4} r(p_{\alpha})$. Now Theorem 3.2 states that there are $k(p_{\alpha})$ and $b(p_{\alpha})$ such that for a suitable $k_{\alpha}\leq k(p_{\alpha})$ we $\rho_{k_{\alpha}, X_{i}}(x) \geq b(p_{\alpha})$. Take $k_{X_{\infty}}$ to be the least integer such that each integer less than or equal to each $k(p_{\alpha})$ divides $k_{X_{\infty}}$. Then Lemma 3.1   implies that a positive lower bound on any $\urho(k_{\alpha}, X_{i})$ gives a positive lower bound on
$\urho(k_{X_{\infty}}, X_{i})$ and the Lemma follows.

\

The same argument, using Lemma 3.1,  shows that, given the statement of Lemma 3.4, there are for each integer $\mu\geq 1$ numbers  $b_{\mu}>0$ (depending only on $X_{\infty}$)such that $\urho(\mu k_{X_{\infty}}, X_{i})\geq b^{2}_{\mu}$ once $i$ is sufficiently large. Now we prove Theorem 1.1 (assuming Theorem 3.2) by contradiction. If Theorem 1.1 is false then there are $X_{i,j}\in {\mathcal K}(n,C,V)$ such that $\urho(X_{i,j}, j!)$ tends to zero for fixed $j$ as $i\rightarrow \infty$.  By Gromov's Compactness theorem  there is no loss in supposing that, for each fixed $j$, the  $X_{i,j}$ converge to some limit $X_{j}$ as $i\rightarrow \infty$. Taking a subsequence $j(\nu)$ we can suppose also that the $X_{j(\nu)}$ converge to $X_{\infty}$. For large enough $\nu$ the integer $k_{X_{\infty}}$ divides $j(\nu)!$; say
$  j(\nu)!= m(\nu) k_{X_{\infty}}$. Now choose $i(\nu)$ so large that $X_{i(\nu), j(\nu)}$ converge to $X_{\infty}$ as $\nu\rightarrow \infty$ and also so that
$\urho(X_{i(\nu), j(\nu)}, j(\nu)!)< b_{\mu(\nu)}$. This gives a contradiction.

\subsection{Proof of Theorem 3.2}
\subsubsection{Cut-offs}
To begin we fix some sequence $k_{\nu}\rightarrow \infty$ so that the scalings of the based space $(X_{\infty}, p)$ by $\sqrt{k_{\nu}}$ converge to a tangent cone $ C(Y)$. For a while we focus attention on this cone. Write $\vert z\vert$ for the distance from the vertex. Let $\Sigma_{Y}\subset Y$ be the singular set and $Y^{\reg}= Y\setminus \Sigma$.

The only information about the singular set which we need is contained in the following proposition. This is very likely a standard  fact but the proof is quite short so we include it. 
\begin{prop}
For any $\eta>0$ there is a function $g$ on $Y$, smooth on $Y^{\reg}$, supported in the $\eta$-neighbourhood of $\Sigma_{Y}$, equal to $1$ on some neighbourhood of $\Sigma_{Y}$ and with
$$  \Vert \nabla g \Vert_{L^{2}}\leq \eta. $$
\end{prop}

Recall that $Y$ has dimension $2n-1$. For clarity in this proof we write $N=2n-1$. A simple argument using the  noncollapsing condition (1.2) and the Bishop inequality in the original manifolds shows that there are fixed numbers
$\oc,\uc>0$ such that for $r\leq 1$ and any metric  ball $B_{r}$ in $Y$ we have
\begin{equation}  \uc r^{N}\leq  \Vol(B_{r})\leq \oc r^{N}.\end{equation}
We know that $\Sigma_{Y}$ is a compact set of Hausdorff dimension strictly less than $N-2$. By the definition of Hausdorff dimension we can find a number $\lambda\in (0,N-2)$ with the following property. For any $\epsilon>0$ there is a cover of $\Sigma_{Y}$ by a finite number of balls $B_{r_{i}/2}(p_{i})$ such that \begin{equation} \sum r_{i}^{N-2-\lambda} <\epsilon. \end{equation}
We write $B_{i}=B_{r_{i}}(p_{i})$ so, in  an obvious notation, the cover is by the balls $\frac{1}{2} B_{i}$. By the Vitali argument we can suppose that the balls $\frac{1}{10}B_{i}$ are disjoint. We take $\epsilon<1$ so for each $i$ we obviously have $r_{i}\leq \epsilon^{1/(N-2-\lambda)}<1$.

Let $\phi(t)$ be a standard cut-off function, vanishing for $t\geq 2$, equal to $1$ when $t\leq 1$ and with derivative bounded by $2$. Define
$$   f_{i}(y)= \phi(r_{i}^{-1}d(y,p_{i})). $$
Thus $f_{i}$ is supported in $B_{i}$ and equal to $1$ in $\frac{1}{2} B_{i}$.
This function need not be smooth but it is Lipschitz and differentiable almost everywhere, with $\vert \nabla f_{i}\vert \leq 2r_{i}^{-1}$. Set $f=\sum f_{i}$. Let $\Psi(t)$ be a cut-off function, equal to  $1$ when $t\geq 9/10$,  with $\Psi(0)=0$ and with derivative bounded by $2$. Put
 $g_{0}= \Psi\circ f$. Then $g_{0}$ is equal to $1$ on a neighbourhood of $\Sigma_{Y}$ and is supported in the $2\epsilon^{1/N-2-\lambda}$-neighbourhood of $\Sigma_{Y}$. Also we have
$$  \Vert \nabla g_{0} \Vert_{L^{2}} \leq 2 \Vert \nabla f \Vert_{L^{2}}. $$
We claim that $\Vert \nabla f\Vert^{2}_{L^{2}}\leq C_{5} \epsilon$ for some fixed $C_{5}$, depending only on $\uc,\oc$. Given this claim we can make $\Vert \nabla g_{0}\Vert_{L^{2}}$ as small as we please and then finally approximate $g_{0}$
by a smooth function $g$ to achieve our result. (Note that this approximation only involves working over a compact subset of $Y^{\reg}$. )

To establish the claim, divide the index set into subsets 
$$  I_{\alpha}=\{i: 2^{-\alpha-1}\leq r_{i}< 2^{-\alpha}\}, $$
for $\alpha\geq 0$. A simple packing argument, using the fact that the balls $\frac{1}{10} B_{i}$ are disjoint, shows that there is a fixed number $C_{6}$ with the following property. If $j\in I_{\alpha}$ then  for each fixed $\beta\leq\alpha$ there are at most $C_{6}$ balls $B_{i}$ with $i\in I_{\beta} $ which intersect $B_{j}$. Now we have
$$ \Vert \nabla f \Vert^{2}_{L^{2}}\leq \sum_{i, j} \int \vert \nabla f_{i}\vert \ \ \vert\nabla f_{j} \vert. $$
Thus $$  \Vert \nabla f \Vert^{2}_{L^{2}}\leq 2\sum_{i,j: r_{j}\leq r_{i}} \int \vert \nabla f_{i}\vert \ \vert\nabla f_{j} \vert. $$
For fixed $j$ there are at most $C_{6} (1+ \log_{2}(r_{j}^{-1}))$ terms which contribute to this last sum. For each term $|\nabla f_{i}|\leq 2 r_{i}^{-1}\leq 2 r_{j}^{-1}$ and the integrand is supported on the ball $B_{j}$ of radius
$2r_{j}$. So for fixed $j$ the contribution to the sum is bounded by
$$8 C_{6}  (1+\log_{2}(r_{j}^{-1}))r_{j}^{-2} \oc (2r_{j})^{N}. $$
Hence, summing over $j$,
$$ \Vert \nabla f \Vert^{2}_{L^{2}}\leq 2^{N+3} C_{6} \oc \sum r_{j}^{N-2} (\log_2 (r_{j}^{-1})+1). $$
We can find a number $C_{7}$ such that for $t \geq 1 $ we have $1+ \log t\leq C_{7} t^{\lambda}$. Thus
$$ \Vert \nabla f \Vert^{2}_{L^{2}}\leq 2^{N+3} C_{6} C_{7} \oc \sum r_{j}^{N-2-\lambda}\leq 2^{N+3} C_{6} C_{7} \oc \epsilon. $$

\

We pick some base  point $y_{0}$ in $Y^{\reg}$. We will need 4 parameters
$\rho, \epsilon,\delta, R$ in our basic construction, where $\rho,\epsilon, \delta$ will be \lq\lq small'' and $R$ \lq\lq large''. In particular $\delta<<\rho<<1<<R$.

First we fix $\rho$ so that $\exp (-\rho^{2}/4)\geq 3/4$ and $\rho \leq  (16 K_{1})^{-1}$ where $K_{1}$ is the constant in our first derivative estimate.
We take $u_{*} =  \rho y_{0}\in C(Y) $, with the obvious notation. Fix any neighbourhood $D$ of $u_{*}$ whose closure does not meet the singular set in $C(Y)$.  For any $\epsilon$ let $Y_{\epsilon}$ be the set of points of distance greater than $\epsilon$ from $\Sigma$. Let $U_{\epsilon,\delta, R}$ be the set of points $z$ in $ C(Y_{\epsilon})$ such that $\delta<\vert z\vert <R$. We choose the parameters so that $U_{\epsilon,\delta,R}$ contains the closure of  $D $. We consider a smooth compactly supported cut-off function $\beta$ on $U_{\epsilon, \delta, R}$. For such a function we set
$$  E_{\beta}= \int e^{-\vert z \vert^{2}/2} \vert \nabla \beta \vert^{2}. $$
\begin{lem} For any given $\zeta>0$ we  can choose $\epsilon, \delta, R$ and a compactly supported function $\beta$ as above such that
\begin{itemize} \item $\beta=1$ on $D$;
\item  $ E_{\beta} \leq \zeta$.
\end{itemize}
\end{lem}

To see this we take $\beta=\beta_{\delta} \beta_{R}\beta_{\epsilon}$ where
\begin{itemize}
\item $\beta_{\delta}$ is a standard cut-off function of $\vert z\vert$ equal to $1$ for $\vert z \vert>2\delta$.
\item $\beta_{R}$ is likewise a standard cut-off function of $\vert z\vert$, equal to $1$ for $\vert z \vert< R/2$.

\item $\beta_{\epsilon}=1-(g \circ \varpi)$ where $g$ is a function on $C(Y)$ of the kind constructed in Proposition (3.5) and $\varpi$ is the radial projection from the cone minus the vertex to $Y$. 
\end{itemize}

\

Then the lemma follows from elementary calculations. 

\

\subsubsection{The topological obstruction}

Recall that the metric on the regular part of the cone has the form $\frac{i}{2}\partial \db \vert z\vert^{2}$. So, just as in the case of $\bC^{n}$, we have a line bundle $\Lambda_{0}$ with connection $A_{0}$, curvature the K\"ahler form $\Omega_{0}$ and  a holomorphic section $\sigma_{0}$ with $\vert \sigma_{0}\vert=\exp(-\vert z\vert^{2}/4)$.  Then $\sigma=\beta \sigma_{0}$ is holomorphic on $D$. Note that  $\Vert \sigma \Vert_{L^{2}}^{2}$ will now be slightly less than $\kappa^{1/2} (2\pi)^{n/2}$ where $\kappa\leq 1$ is the volume ratio as in (2.1).

As we explained, there certainly is some constant C giving the elliptic estimate (H3) and we use the Lemma to choose $\epsilon, \delta,R$ so that this set of data has Property (H).

The parameters $\rho,\delta,\epsilon, R$ are now all {\it fixed}. We set $U=U(\rho,\delta,\epsilon,R)$. 

\

\
Consider now a $C^{0}$-small perturbation $g,J$ of the metric and complex structure $g_{0}, J_{0}$, and hence a perturbation $\Omega$ of $\Omega_{0}$ We suppose that $-i\Omega$ is the curvature of a unitary connection $A$ on a bundle $\Lambda$. If we can choose a bundle isomorphism between $\Lambda$ and $ \Lambda_{0}$ such that, under this isomorphism, the connection $A$ is a small perturbation of $A_{0}$ then we can apply Proposition 2.4 to conclude that the data $J,\Omega, A$ also has Property (H), (for suitably small perturbations). The difficulty is that if $H_{1}(U;\bZ)\neq 0$ a connection on a line bundle is not determined by its curvature. Said in another way, we consider the line bundle $\Lambda\otimes \Lambda_{0}^{*}$ with the  connection $a$ induced from  $A,A_{0}$. The curvature of $a$ is small but $a$ need not be close to a trivial flat connection. 
There is no real loss of generality in supposing that $Y_{\epsilon}$ has smooth boundary ( because we can always replace it by a slightly enlarged domain). Write $\underline{\nu}$ for the normal vector field on the boundary. We want to recall some Hodge Theory on this manifold with boundary. Fix $p>2n$.
. 

\begin{prop}
\begin{enumerate}\item The infimum of the $L^{2}$ norm on the closed $2$-forms in a cohomology class defines a norm on $H^{2}(Y_{\epsilon}, \bR)$. 
\item Define ${\mathcal H}^{1}$ to be the set of 1-forms $\alpha $ on $\overline{Y_{\epsilon}}$ with $d\alpha=0, d^{*}\alpha=0$ and with $( \alpha, \underline{\nu})=0$ on the boundary. Then the natural map from ${\mathcal H}^{1}$ to $H^{1}(Y_{\epsilon}, \bR)$ is an isomorphism. 
\item If $F$ is any exact $2$-form on $\overline{Y_{\epsilon}}$ there is a unique $1$-form $\alpha$ such that $d^{*}\alpha=0, d\alpha=F, (\alpha, \underline{\nu})=0$ and $\alpha$ is $L^{2}$-orthogonal to ${\mathcal H}^{1}$. We have, for some fixed constant $C_{8}$,
$\Vert \alpha \Vert_{L^{p}_{1}}\leq C_{8} \Vert F \Vert_{L^{p}}$.
\end{enumerate}
\end{prop}

These are fairly standard results. The first item follows from the fact that the $L^{2}$ extension of the image of $d$ is closed. The second asserts the unique solubility of the Neumann boundary value problem for the Laplacian on functions on $Y_{\epsilon}$. The existence and uniqueness of $\alpha$ in the third item is similar. The $L^{p}$ estimate in the third item follows from general theory of elliptic boundary value problems, see \cite{Wehrheim} for a detailed treatment of this case. {Note that in our application the subtleties of the boundary value theory could be avoided by working on a slightly larger domain. Then we can reduce to easier interior estimates. Alternatively one can adjust the set-up to reduce to the standard Hodge theory over a compact \lq\lq double''.)

Write $a\vert$ for the restriction of the connection $a$ to the restricted bundle over $Y_{\epsilon}$.
A consequence of item (1) is that there is some number $C_{9}>0$ such that any closed $2$-form $F$ over $Y_{\epsilon}$ which represents an integral cohomology class and  with $\Vert F\Vert_{L^{2}}\leq C_{9}$ is exact. In particular we can apply this to the curvature $F_{a\vert}=i(\Omega-\Omega_{0})$ of the connection $a\vert$, using the fact that this represents an integral class. (Here we are considering $Y_{\epsilon} $ as embedded in $U$ in the obvious way.)
Thus there is a $C_{10}>0$ such that if $\Vert \Omega-\Omega_{0}\Vert_{U}\leq C_{10}$ we can apply item (3) of Prop. 3.7 to write $F_{a\vert}=d\alpha$ over $Y_{\epsilon}$ for a small $\alpha=\alpha(a)$. More precisely, $\alpha$ is small in $L^{p}_{1}$ and so in $C^{0}$ by Sobolev embedding. Then $a\vert-\alpha$ is a flat connection on the restriction of $\Lambda\otimes \Lambda_{0}^{*}$ to $Y_{\epsilon}$. This flat connection is determined up to isomorphism by its holonomy: a homomorphism from $H_{1}(Y_{\epsilon}, \bZ)$ to $S^{1}$. 

 Fix a direct sum decomposition of $H_{1}(Y_{\epsilon}, \bZ)$ into torsion and free subgroups. 
Then we get
$$ {\rm Hom}(H_{1}( Y_{\epsilon}, \bZ), S^{1}) = G \times T, $$
where $G$ is a finite abelian group and $T=H^{1}(Y_{\epsilon}, \bR)/H^{1}(Y_{\epsilon}, \bZ)$ is a torus. (We will write the group structures multiplicatively.) 
Thus for our connection $a$ with suitably small curvature we get two invariants
$  g(a)\in G, \tau(a)\in T$. If both vanish then the restriction of the connection to $Y_{\epsilon}$ is close to the trivial flat connection. When $a$ is the connection induced from $A, A_{0}$ as above we write $g(A, A_{0}), \tau(A,A_{0})$. 

\begin{prop} We can find a neighbourhood $W$ of the identity in $T$ and a number $\psi>0$ to the following effect. If $g, J, A$ is a set of data on $U$ with
\begin{itemize} \item $$\Vert g-g_{0}\Vert_{U} \leq \psi, \Vert J-J_{0}\Vert_{U} \leq \psi;$$
 \item $g(A,A_{0})=1$;
\item $ \tau(A, A_{0})\in W$;
\end{itemize}
then$(g,J, A)$ has Property (H).

\end{prop}

\

This is straightforward. The hypotheses imply that, for small $W,\psi$, there is a trivialisation
of $\Lambda\otimes\Lambda_{0}^{*}$ over $Y_{\epsilon}$ in which the connection form is small in $L^{p}_{1}$ and hence in $C^{0}$.  Then extend this to a trivialisation over $U$ by parallel transport along rays. In this trivialisation the radial derivative of the connection form is given by a component of the curvature, so is controlled by $\psi$. From another point of view this trivialisation is a bundle isomorphism between $\Lambda, \Lambda_{0}$ under which $A$ is a small perturbation of $A_{0}$.

\

Let $m_{1}$ be the order of $G$. Thus for any $g\in G$ we have $g^{m_{1}}=1$. Fix a slightly smaller neighbourhood $W'\subset\subset W$ of the identity in $J$. By Dirichlet's theorem we can find an $m_{2}$ such that for any $\tau\in T$ there is a power $\tau^{q}$ which lies in $W'$ where $1\leq q\leq m_{2}$. Write $m=m_{1} m_{2}$. Now return to our connection    $a\vert$  on the bundle $\Lambda\otimes \Lambda_{0}^{*}$ over  $Y_{\epsilon}$. Recall that for integer $t$ we write $a\vert^{\otimes t}$ for the induced connection on $\Lambda^{ t}\otimes \Lambda_{0}^{-t}$ over $Y_{\epsilon}$. Suppose that
  $\Vert F(a\vert)\Vert_{U} \leq C_{10}/m$. Then for $1\leq t\leq m$ the invariants
  $g(a\vert^{\otimes t}), \tau(a\vert^{\otimes t})$ are defined and we have:
 \begin{prop} We can choose $t$ with $1\leq t\leq m$ such that $g(a\vert^{\otimes t})=1$ and $\tau(a\vert^{\otimes t})\in W'$.
\end{prop}

\

 With $m$ fixed as above, write $$\tilde{U}= U(m^{-1/2}\delta, \epsilon, R). $$
  For integers $t$ with $1\leq t\leq m$ let $\mu_{t}:U\rightarrow \tilde{U}$ be the map $\mu_{t}(z)= t^{-1/2} z$ (in obvious notation). Thus $\mu_{t}^{*}(t \Omega_{0})=  \Omega_{0}$.
 
 Our model structure $g_{0}, J_{0}, \Lambda_{0}, A_{0}$ is defined over $\tilde{U}$ Now consider deformed structures $J, \Omega, \Lambda, A $ as before but which are also defined over $\tilde{U}$.  Suppose that
$$  \Vert g-g_{0}\Vert_{\tilde{U}}\leq \tilde{\psi}, \Vert J-J_{0}\Vert_{\tilde{U}}\leq \tilde{\psi},
$$  where $\Vert \ \Vert_{\tilde{U}}$ here denotes $C^{0}$ norms over $\tilde{U}$. For integers $t$ as above, let $g_{t}, J_{t}, \Lambda_{t}, A_{t}$ be the data over $U$ given by pulling back $ tg, J, \Lambda^{t}, A^{\otimes t}$ using the map $\mu_{t}$. 
It is clear that if 
 $\tilde{\psi}$ is sufficiently small then for every $t$ we have

$$  \Vert  g_{t}-g_{0}\Vert_{U}\leq \psi\ , \  \Vert J_{t}-J_{0}\Vert_{{U}}\leq \psi. $$
 It is also clear that, if $\tilde{\psi}$ is sufficiently small, then the invariants $g(A_{0}, A_{t}), \tau(A_{0}, A_{t})$ are defined. 
\begin{prop}
If $\tilde{\psi}$ is sufficiently small then we can choose $t\leq m$ so that $g(A_{0}, A_{t})=1$ and $\tau(A_{0}, A_{t})\in W$. 
\end{prop}

We choose $t$ according to Prop. 3.9, so that $g(a\vert^{\otimes t})=1$ and $\tau(a\vert^{\otimes t})\in W'$.

Write $\tau(a\vert^{\otimes t})=\tau$. Thus $\tau$ can be regarded as a small element of $H^{1}(Y_{\epsilon}, \bR)$. It follows from our set-up that there is a trivialisation of the bundle $\Lambda^{t}\otimes \Lambda_{0}^{-t}$ over $Y_{0}$ in which the connection $a\vert^{\otimes t}$ is represented by  a $C^{0}$-small connection form. Extend this trivialisation to $\tilde{U}$ using parallel transport along rays. As above, in the proof of Proposition 3.8, the radial derivative of the connection form in this trivialisation is given by the curvature $F_{a^{\otimes t}}$ and it follows easily that if $\tilde{\psi}$ is sufficiently small  then in the induced trivialisation the pull-back $\mu_{t}^{*}(a^{\otimes t})$, restricted to $Y_{\epsilon}$ has a $C^{0}$-small connection form. In particular, given that $W'\subset \subset W$ we can, by fixing $\tilde{\psi}$ sufficiently small, ensure that the \lq\lq $\tau$ invariant'' of this connection lies in $W$ and the \lq\lq g-invariant'' is $1$.   Now the fact that $\mu_{t}^{*}(A_{0}^{\otimes t})$ is isomorphic to $A_{0}$ yields the result stated.

\

We sum up in the following way. 
\begin{prop} We can choose $\tilde{\psi}>0$ to the following effect.
Suppose $g,J,  \Lambda, A $ are structures as above over  $\tilde{U}$. Suppose that $\Vert g-g_{0}\Vert_{\tilde{U}}, \Vert J-J_{0}\Vert_{\tilde{U}}\leq \tilde{\psi}$. Then we can find an integer $t$ with $1\leq t\leq m$ such that the data  $\mu_{t}^{*}(t g), \mu_{t}^{*}(J), \mu_{t}^{*}(\Lambda^{t}), \mu_{t}^{*} (A^{\otimes t})$ over  $U$ has Property(H). 
\end{prop}

\subsubsection{Completion of Proof}

With this lengthy discussion involving the tangent cone in place, we return to the limit space $X_{\infty}$. Recall that we have a sequence of scalings $\sqrt{k_{\nu}}$. We consider embeddings
$\chi_{\nu}: \tilde{U}\rightarrow X^{\reg}_{\infty}$.
Given such a $\chi_{\nu}$ we write $J^{\nu}$ for the pull-back of the complex structure on $X_{\infty}^{\reg}$ and $g^{\nu}$ for the pull-back of $k_{\nu}$ times the metric.

\begin{prop}
There is a $k_{\nu}$ so that we can find an embedding $\chi_{\nu}$ as above, such that
\begin{itemize}
\item $$ (1/2) k_{\nu}^{-1/2}\vert  z\vert \leq  d(p, \chi_{\nu}(z))\leq 2 k_{\nu}^{-1/2} \vert z \vert; $$ 
\item
$$  \Vert J^{\nu} -J_{0} \Vert_{\tilde{U}}, \Vert g^{\nu}-g_{0}\Vert_{\tilde{U}} \leq \tilde{\psi}/2. $$
\end{itemize}
\end{prop}
This follows easily from the general assertions in Section 2.1 about convergence. We now fix this $k_{\nu}$ and define $k(p)= m k_{\nu}$
and $r(p)=  \rho k(p)^{-1/2}$. We write $\chi_{k_{\nu}}=\chi$.  

\

Let   $X_{i}\in {\mathcal K}(n,C,V)$ be a sequence  converging to $X_{\infty}$. We fix distance functions on $X_{\infty}\sqcup X_{i}$.
  We consider embeddings $\chi^{i}:\tilde{U}\rightarrow X_{i}$. Given such maps we write $g_{i}, J_{i}$ for the pull backs of the metric and complex structure, $\Lambda_{i}$ for the pull-back of $L^{k_{\nu}}$ and $A_{i}$ for the pulled back connection.
\begin{prop}
For large enough $i$ we can choose $\chi^{i}$ with the following two properties.
\begin{itemize}
\item $ d(\chi^{i}(z), \chi(z)) \leq \rho k(p)^{-1/2}/100$
\item $ \Vert g_{i}-g_{0}\Vert_{\tilde{U}}, \Vert  J_{i}-J_{0}\Vert_{\tilde{U}}\leq \tilde{\psi}$.
\end{itemize}
\end{prop}
Again this follows from our general discussion of convergence.

Fix $i$ large enough, as in Proposition 3.13.
We apply Proposition 3.11 to find a $t$ such that the pull-back by $\mu_{t}$ of the data $t g_{i}, J_{i}, \Lambda_{i}^{t}, A^{\otimes t}_{i} $ has Property (H) over  $U$. Now write $k=t k_{\nu}$ so $k\leq k(p)$. We apply Proposition 2.4  to construct a  holomorphic section $s$ of $L^{k}\rightarrow X_{i}$, with a fixed bound on the $L^{2,\sharp}$ norm and with $\vert s(x)\vert \geq 1/4$ at points $x$ with $d^{\sharp}(x,\chi_{i}(u_{*}))<(4K_{1})^{-1}$. Here we are writing $d^{\sharp}$ for the scaled metric, so in terms of the original metric the condition is $d(x,\chi_{i}(u_{*}))< k^{-1/2} (4K_{1})^{-1}$.

To finish, suppose $q\in X_{i}$ has $d(q,p)\leq r(p)$. By construction $r(p)\leq \rho k^{-1/2}$. Note also that if we set $p'=\chi^{i} (t^{-1/2}z_{\rho})$. then
$$ d(q,p')\leq d(q,p)+ d(p, \chi ( t^{-1/2} z_{\rho}))+ d( \chi ( t^{-1/2} z_{\rho}), \chi^{(i)} ( t^{-1/2} z_{\rho}))\leq  4\rho \sqrt{k}. $$
This means that $d^{\sharp}(q,p')\leq 4 \rho$ which is less than $(4 K_{1})^{-1}$ by our choice of $\rho$.
\section{Connections with algebraic geometry}

 The  consequences of Theorem 1.1, for the relation between algebro-geometric and differential geometric limits, could be summarised by saying that things work out in the way that one might at first sight guess at.  As we have mentioned before, the proofs  of many of the statements, given Theorem 1.1, have been outlined by Tian in \cite{Tian3}. Thus we view this Section, broadly speaking, as an opportunity to attempt a careful exposition of the material. 
\subsection{Proof of Theorem 2}

\begin{lem} There are numbers  $N_{k}$, depending only on $n,c,V,k$, such that for any $X$ in ${\mathcal K}(n,c,V)$  we have ${\rm dim} H^{0}(X;L^{k})\leq N_{k}+1$.
\end{lem}

We work in the rescaled metric. Given $\epsilon>0$ we can choose a maximal set of points $x_{i}$ in $X$ such that the distance between any two is at least $\epsilon$. Then the $2\epsilon$ balls with these centres cover $X$ and the $\epsilon/2$ balls are disjoint. 
Consider the evaluation map
$$  {\rm ev}: H^{0}(X;L^{k} )\rightarrow \bigoplus L^{k}_{x_{i}}. $$
We first show that if $\epsilon$ is sufficiently small then this map is injective.
For if it is not injective  there  is a holomorphic section $s$ with $L^{2,\sharp}$ norm $1$ vanishing at all the $x_{i}$. Since the $2\epsilon$ balls cover we get $\Vert s\Vert_{L^{\infty}} \leq 2K_{1} \epsilon$. This gives a contradiction to $\Vert s\Vert_{L^{2\sharp}}=1$ if $\epsilon$ is small enough. On the other hand since the $\epsilon/2$ balls are disjoint the non-collapsing condition gives an upper bound on the number of the points $x_{i}$ which completes the proof. 

In fact the estimate one gets by this argument  is
$$   N_{k}+1= \frac{2^{4n}K_{1}^{2n} V^{n+1} n!}{c \pi^{n}}k^{n^{2}}, $$
which is very poor compared with the asymptotics we know that ${\rm dim} H^{0}(X,L^{k})\sim (2\pi)^{-n}V k^{n}$ for a fixed $X$, as $k\rightarrow \infty$.

For our purposes there is no loss of generality in supposing that the $k_{0}$ of Theorem 1.1 is $1$. Then the sections of $L^{k}$ define a regular map of $X$ for all $k$. Suppose we choose  isometric embeddings
$$   \phi_{k}: H^{0}(X;L^{k})^{*} \rightarrow \bC^{N_{k}+1}, $$
using the $L^{2}$ norm on the left hand side and the fixed standard Hermitian form on the right. Then we get  projective varieties
$$   V(X,\phi_{k})\subset \bC\bP^{N_{k}}, $$
and holomorphic maps $$  T_{k}: X\rightarrow V(X,\phi_{k}).$$
Of course $T_{k}$ depend on the choice of $\phi_{k}$ which is arbitrary, but any two choices differ by the action of the unitary group $U(N_{k}+1)$. The fact that this group is {\it compact} will mean that in the end the choice of $\phi_{k}$ will not be important. Soon we will reduce to the case when $T$ is generically 1-1 but we do not need to assume that yet, so  $T_{k}$ could map to a variety of dimension less than $n$ or be a multiple cover of an $n$-dimensional variety.  In any case we get, by straightforward arguments, a fixed upper bound on the {\it degree} of $V(X,\phi_{k})$ (depending on $k,n,V$). 

By standard general principles there is a system of morphisms of projective varieties, for integer $\lambda$,
$$   f_{\lambda} :V(X,\phi_{\lambda k})\rightarrow V(X,\phi_{k}), $$
with $f_{\lambda \mu}=f_{\lambda} \circ f_{\mu}$ and $f_{\lambda}T_{\lambda k}= T_{k}.$

Now we bring in the crucial lower bound provided by Theorem 1. 

\begin{lem} Taking $k_{0}=1$, the map $T_{1}:X\rightarrow V(X,\phi_{1})$ has derivative bounded by $K_{1} b^{-1}$ where $b$ is the lower bound in Theorem 1.1 and $K_{1}$ is the constant in the first derivative estimate.
\end{lem}
Here we are referring to the \lq\lq operator norm'' of the derivative, regarded as a map from the tangent space of $X$ at a point, with the given metric $g$, to the tangent space of $\bC\bP^{N_{1}}$ with the standard Fubini-Study metric.

The proof of  the Lemma comes directly from the definitions. Given a point $x\in X$ we can choose an orthonormal basis of sections $s_{0}, s_{1}, \dots, s_{N} $ with $s_{i}(x)=0$ for $i>0$ and $\vert s_{0}(x)\vert = B\geq b$. There is no loss of generality in supposing that $\phi_{1}$ maps the dual basis to the first $N+1$ basis vectors in $\bC^{N_{1}+1}$. Fix a unitary isomorphism of the fibre $L_{x}$ with $\bC$. Then the derivative of each $s_{i}$, for $i>0$ can be regarded as an element of the cotangent space of $X$ at  $x$. Identifying the tangent space of $\bC\bP^{N_{1}}$ at
$(1,0,\dots,0)$ with $\bC^{N_{1}}$ in the standard way, the derivative of $T$ at $x$ is represented by
$$    B^{-1} ( \partial s_{1}, \dots, \partial s_{N},  \dots,0), $$
and the lemma follows.

Using  Lemma 3.1 we get similar universal bounds on the derivatives of all maps $T_{k}$, for suitable constants which we do not need to keep track of.

Now suppose that $X_{i}$ is a sequence in ${\mathcal K}(n,c,v)$ with Gromov-Hausdorff limit a polarised limit space $X_{\infty}$.   For each fixed $k$ we choose $\phi_{k,i}$ so we have a sequence of projective varieties $V(X_{i},\phi_{k,i})$ of bounded dimension and degree. By standard results we can, choosing a subsequence suppose that for each $k$ these converge in the algebro-geometric sense to a limit $W_{k}$. (More precisely, we can suppose that for each $k$ the $V(X_{i},\phi_{k,i})$
have fixed degree and dimension and converge as points in the Chow variety parametrising algebraic cycles of that type. Then we take $W_{k}$ to be the corresponding algebraic set.) It follows easily from the compactness of 
$U(N_{k}+1)$ that $W_{k}$ is independent, up to  projective unitary transformations, of the choice of maps $\phi_{k,i}$. 
\begin{lem} After perhaps passing to a subsequence  of the $X_{i}$, for each $k$ the maps $T_{k}:X_{i}\rightarrow V(X_{i}, \phi_{k})$ extend by continuity to a  continuous map $T_{k}:X_{\infty}\rightarrow W_{k}$, holomorphic on $X^{\reg}_{\infty}$.
\end{lem}

More precisely what we mean is that we suppose we have fixed metrics on the $X_{i}\sqcup X_{\infty}$ then for all $\epsilon>0$ we can find $\delta>0$ so that  the distance in the projective space between $T_{k}(y), T_{k}(y)$ is less than $\epsilon$ if $d(x,y)<\delta$. 

The proof of the Lemma is very easy using the equicontinuity of the maps $T_{k}$ on the $X_{i}$. The limit map $T_{k}$ on $X_{\infty}$ is unique up to unitary transformations preserving $W_{k}$ and the possible existence of such maps is the only reason that we may need to pass to a subsequence.

In the next subsection we will collect some further analytical results which will give a much clearer view of the situation. Then we return to discuss the relation between $X_{\infty}$ and the $W_{k}$ further in  subsection 4.3.

\subsection{More analysis}
Recall that we have a uniform $C^{0}$ estimate (Prop. 2.1) for holomorphic sections of $L^{k}\rightarrow X$, for any $X$ in ${\mathcal K}(n,c,V)$. We will now extend this to a polarised limit space $X_{\infty}$.
\begin{lem}
If $s$ is a bounded holomorphic section of $L^{k}$ over $X_{\infty}^{reg}$ then
   $$\Vert s\Vert_{L^{\infty}} \leq K_{0} \Vert s\Vert_{L^{2,\sharp}}. $$
\end{lem}
Here of course we are writing $L^{k}\rightarrow X^{\reg}_{\infty}$ for the limiting line bundle and we are defining the $L^{2,\sharp}$ norm with the rescaled metric.

We prove the Lemma by contradiction. The argument is very similar to our main construction in Section 3. Suppose there is a  holomorphic section $s$
with $\Vert s \Vert_{L^{2,\sharp}}= 1, \Vert s \Vert_{L^{\infty}}= B $ and there is a point $p \in X^{\reg}_{\infty}$ with $\vert s(p)\vert=K_{0}+\lambda$ for some $\lambda>0$. Choose a neighbourhood $D$ of $p$ which lies inside $X_{\infty}$. There is some constant $C$ so that an estimate like that in (H3) of Property (H) holds. The singular set in $X_{\infty}$ has Hausdorff codimension strictly bigger than $2$ so by the argument of Prop. 3.5 we can construct a cut-off function $\beta$ equal to $1$ over $D$ and with $\Vert \nabla \beta\Vert_{L^{2}}$ as small as we like. In particular we can make this much smaller than $\lambda B^{-1}C^{-1}$. When $i$ is large we can choose maps $\chi_{i}$ from a neighbourhood of the support of $\beta$ into $X_{i}$ and lifts $\hat{\chi}_{i}$ so that the structures match up as closely as we please. Transport $\beta s$ by these maps to a section of $L^{k}\rightarrow X_{i}$ and adjust to get a holomorphic section $s_{i}$ just as in Section 3. Then when $i$ is large enough we see that $s_{i}$ contradicts Prop. 2.1, by arguments just like those in Section 3. 

\
\

Now we {\it define} $H^{0}(X_{\infty}, L^{k})$ to be the space of bounded holomorphic sections over the regular part. Let $S\subset \bR$ be the set
$$   S=\{0,1,1/2,1/3,1/4,\dots 1/i\dots \}, $$
and for integers $j$ let $S_{j}\subset S$ be the subset $\{0, j^{-1}, (j+1)^{-1}\dots\}$. We are regarding $S$ as a topological space, so any sequence tending to zero would do equally well. Let $${\mathcal X}= \bigsqcup_{i=1,2, \dots, \infty} X_{i}. $$
Thus there is a map of sets $\pi: {\mathcal X}\rightarrow S$ which takes $X_{i}$ to $i^{-1}$ for $i=1,\dots \infty$. 
The distance functions on $X_{i}\sqcup X_{\infty}$ define a natural topology on ${\mathcal X}$ such that $\pi$ is continuous.

Now let $${\mathcal H} = \bigsqcup_{i=1,2,\dots,\infty} H^{0}(X_{i},L^{k}), $$
taking the above definition in the case $i=\infty$. There is an obvious map of sets
$\varpi:{\mathcal H}\rightarrow S$. We put a topology on ${\mathcal H}$ by saying that sections are close if they are close when compared by maps $\chi_{i}, \hat{\chi}_{i}$, as above.

\begin{lem}
For sufficiently large $j$ the restriction of $\varpi:{\mathcal H}\rightarrow S$ to $S_{j}\subset S$ is a vector bundle.
\end{lem}

(Note that this is for fixed $k$: for different values of $k$ one might {\it a priori} have to take different values of $j$.)

The proof uses much the same construction as in Lemma 4.6. The content of the statement is that, for large enough $i$, we can define linear isomorphisms
$$   Q_{i}: H^{0}(X_{\infty}, L^{k})\rightarrow H^{0}(X_{i}, L^{k})$$
such that $Q_{i}(s)$ tends to $s$ as $i\rightarrow \infty$, in the sense above. We choose a family of compactly supported cut-off functions $\beta_{i}$ on $X^{\reg}_{\infty}$  with the following properties. 
\begin{itemize}
\item The compact sets $\beta_{i}^{-1}(1)$ give an exhaustion of $X_{\infty}^{\reg}$;
\item The support of $\beta_{i}$ is contained in the domain of a map $\chi_{i}$ under which the structures compare with a small error $\eta_{i}$ with $\eta_{i}\rightarrow 0$ as $i\rightarrow \infty$;
\item $\Vert \nabla \beta_{i}\Vert_{L^{2,\sharp}}\rightarrow 0$ as $i\rightarrow \infty$. In particular $\Vert \nabla \beta_{i}\Vert_{L^{2,\sharp}}$ can be taken very small compared with $K_{0}^{-1}$. 
\end{itemize}

  Then for any holomorphic section $s\in H^{0}(X_{\infty}, L^{k})$ we transport $\beta_{i} s$ to $X_{i}$ using $\chi_{i}$ and project to get an element $Q_{i}(s)\in  H^{0}(X_{i}, L^{k})$ in the familiar way. Our standard argument shows that
 $Q_{i}(s)$ can be made as close as we please to $s$ by taking $i$ large. In particular this shows that $Q_{i}$ is injective, for large $i$. (Note that the point of establishing Lemma 4.4 first is that  the bounds we require on $\Vert \nabla \beta_{i}\Vert $ do not depend on $s$, but only on $K_{0}$.) To prove surjectivity we argue by contradiction. If $Q_{i}$ is not surjective we can find $s_{i}\in H^{0}(X_{i}, L^{k})$ of $L^{2,\sharp}$ norm $1$ and $L^{2,\sharp}$-orthogonal to the image of $Q_{i}$. Passing to a subsequence and taking a limit as $i\rightarrow \infty$ we get a section $s_{\infty}\in H^{0}(X_{\infty}, L^{k})$. The $C^{0}$ estimate shows that $s_{\infty}$ has $L^{2,\sharp}$ norm $1$ and we easily get a contradiction to the fact that $s_i$ is orthogonal to $Q_{i}(s_{\infty})$ for all $i$.

Our reason for formulating things in this way  is that it is natural to consider families $\pi:{\mathcal X}\rightarrow B$ over a general base. Here we want the fibres of $\pi$ to be either smooth manifolds in ${\mathcal K}(n,c,V)$ or polarised Gromov-Hausdorff limits of such, and we want the topology on ${\mathcal X}$ to be compatible with the Gromov-Hausdorff distance in an obvious way. It is not hard to set up the definitions and the proof of Lemma 4.5 shows that, if $B$ is connected, there is a \lq\lq direct image'' which is a vector bundle over $B$. However there does not seem much point in developing the theory in detail since in the end, after we have proved Theorem 1.2, this construction can be obtained from the standard algebraic geometry direct image.

\

\

We now turn to the problem of separating points. 
\begin{prop}
Suppose $X_{\infty}$ is a polarised limit space and $\rho>0$. We can find a $k$ such that if $p_{1}, p_{2}\in X_{\infty}$ are points with $d(p_{1}, p_{2})>\rho $ then the map $T_{k}:X_{\infty} \rightarrow \bC\bP^{N_{k}}$ takes  $p_{1}, p_{2}$ to distinct points in $\bC\bP^{N_{k}}$.
\end{prop}
The proof is a small extension of our main argument in Section 3. By a compactness argument, it suffices to find a $k$ which works for a fixed pair of distinct points $p_{1}, p_{2}$. We choose a sequence $X_{i}$ from ${\mathcal K}(n,c,V)$ converging to $X_{\infty}$. We can find a sequence $k_{\nu}\rightarrow \infty$ such that rescaling $X_{\infty}$ by $\sqrt{k_{\nu}}$ at each of the points we get convergence to tangent cones $C(Y_{1}), C(Y_{2})$ and we construct $U_{1}, U_{2}$ etc. in each case. We then choose $k$ so that we get maps $\chi_{s}:U_{s}\rightarrow X_{i}$ as in Section 3. Clearly we can also suppose that $\chi_{1}(U_{1}), \chi_{2}(U_{2})$ are disjoint. (For this we will need to take $\sqrt{k}$ large compared with $\rho^{-1}$.)  Then we get holomorphic sections $s_{1}, s_{2}$ of $L^{k}\rightarrow X_{i}$ with fixed $L^{2,\sharp}$ norm and such that $\vert s_{i}\vert \geq 1/2$ say at points close to $p_{i}$.
Consider the section $s_{1}$ at points in $X_{i}$ close to the image of $\chi_{2}$.
Recall that $s_{1}= \sigma_{1}-\tau_{1}$ where $\sigma_{1}$ vanishes on the image of $\chi_{2}$ and the $L^{2,\sharp}$ norm of $\tau_{1}$ can be made as we please by our original choice of parameters. Let $u_{*}\in D\subset  U_{2}$ be the base point. Since $\tau_1$ is holomorphic over $\chi_2(D)$ the size of $\tau_1(\chi_2(u_{*}))$ can  be controlled by the $L^{2}$ norm of $\tau_1$ over $\chi_2(D)$. Thus by a suitable choice of original parameters (depending only on knowledge of $Y_{1}, Y_{2}$) we can arrange that $\vert s_1(x)\vert=\vert \tau_1(x)\vert \leq 1/100$, say, for points $x$ close to $\chi_{2}(u_{*})$. Taking the limit as $i\rightarrow \infty$ we get sections $s_{1}, s_{2}\in H^{0}(X_{\infty}, L^{k})$ with $\vert s_{i}(p_{i})\vert\geq 1/2$ and $\vert s_{i}(p_{j})\vert \leq 1/100$ for $i\neq j$.

\begin{prop}
Given a compact set $K\subset X_{\infty}^{\reg}$ we can find an integer $m(K)$ such that for $k\geq m(K)$, any point $x \in K$ and any tangent vector $v$ at $x$ there is a holomorphic section $s\in H^{0}(X_{\infty},L^{k})$ with $s(x)=0$ and the derivative of $s$ along $v$ not zero.
\end{prop}
This is another straightforward application of the H\"ormander technique.

\subsection{Recap}

We can go back to the discussion of (4.1) and state things in a much clearer way. For a given $k$ we can suppose that all the spaces $H^{0}(X_{i}, L^{k})$ have the same dimension and identify them with $H^{0}(X_{\infty}, L^{k})$ as in Lemma 4.5. In the usual way, the sections in $H^{0}(X_{\infty}, L^{k})$ define a holomorphic map from $X^{\reg}_{\infty}$ to $\bP(H^{0}(X_{\infty}, L^{k})^{*})$.
We fix a basis in  $H^{0}(X_{\infty}, L^{k})$ so that we can say that we map to $\bC\bP^{N}$. The same argument as in lemma 4.2 gives a bound on the derivative of this map so it has a unique continuous extension to $X_{\infty} $. Pulling back the hyperplane bundle by this map (in the case $k=1$) defines an extension of the line bundle $L$ to $X_{\infty}$ (at this stage, as a topological bundle). Theorem (1.1) implies that the original metric on $L$ is uniformly equivalent to the metric pulled back from the hyperplane bundle. The convergence of the maps $T_{k}:X_{i}\rightarrow \bC\bP^{N}$ to $T_{k}:X_{\infty}\rightarrow \bC\bP^{N}$ over the regular part is completely clear because of the way we chose our  identifications of $H^{0}(X_{i}, L^{k}), H^{0}(X_{\infty}, L^{k})$. The algebraic set $W_{k}$ is the image $T_{k}(X_{\infty})$. It is also clear that we have a system of morphisms
 $f_{\lambda}: W_{k\lambda}\rightarrow W_{k}$ such that $f_{\lambda\mu}=f_{\lambda}\circ f_{\mu}$ and $$T_{k}= f_{\lambda}\circ T_{k\lambda}:X_{\infty}\rightarrow W_{k}.$$
(The morphism $f_{\lambda}$ can be viewed as induced by the linear map
which is the transpose of $s^{\lambda}(H^{0}(X_{\infty},L^{k}))\rightarrow H^{0}(X_{\infty}, L^{\lambda k})$ composed with the inverse of the Veronese map.)

Suppose we have any  collection of sets $W_{k}$, for integers $k\geq 1$, and maps $f_{\lambda}:W_{k\lambda}\rightarrow W_{k}$ with $f_{\lambda\mu}=f_{\lambda}\circ f_{\mu}$. Then we can form the  limit set $W_{\leftarrow}\subset \Pi_{k} W_{k}$ given by sequences $(w_{1},w_{2}, \dots)$ such that $f_{\lambda}(w_{k\lambda})= w_{k}$ for all $k,\lambda$. If we have another set $X_{\infty}$ and maps $T_{k}: X_{\infty}\rightarrow W_{k}$ compatible with the $f_{\lambda}$ then we get an induced map from $X_{\infty}$ to $W_{\leftarrow}$. In our situation, Proposition 4.6 implies that this  map is a bijection so what we know at this stage is that we can recover the Gromov-Haussdorf limit algebro-geometrically (at least as a set) in this way.  

It is interesting to compare this with  \cite{D1}, \cite{D2} where the first-named author made a different attack on the same kind of problem. This attack was made  in the absence of Theorem 1.1, and the cost of that absence was that one got a system like the $f_{\lambda}$ but only of {\it rational} maps (or \lq\lq web of descendants'' in the language of \cite{D1}). The core of the problem was that, without something like Theorem 1.1, one does not know that the $W_{k}$ are irreducible. This difficulty is also explained by Tian in \cite{Tian3}. The construction of \cite{D1} should probably best be thought of as an attempt to define the Gromov-Hausdorff limit as a \lq\lq limit'' of algebraic sets or schemes (in the sense of ${\rm lim}\leftarrow$) in this fashion. (From a more algebraic point of view the limiting process we conceive of here is related to considering rings that are not finitely generated.) But, having now Theorem 1.1, we can take a  simpler and more direct path (in the context of manifolds satisfying the hypotheses (1.1),(1.2)).  However it seem likely that related ideas on the algebraic side may play a role in the future in the study of constant scalar curvature K\"ahler metrics (lacking (1.1), (1.2)). In this direction, see the recent work of Szekelyhidi \cite{Sz}.

\subsubsection{Completion of proof of Theorem 1.2}

\begin{lem} For each $k$, the algebraic set $W_{k}$ is irreducible.
\end{lem}
This is crucial, as we indicated above, but the proof is easy. The set $X^{\reg}_{\infty}$ is dense in $X_{\infty}$ so its image is dense in $W_{k}$. Thus we can choose a point $x_{0}\in X^{\reg}_{\infty}$ so that $T_{k}(x_{0})$ lies in a unique component $U$ of $W_{k}$. Suppose there is a point $w$ in $W_{k}$ which is not in $U$. Then we can find a polynomial $P$ of degree $\lambda$ say so that $P$ vanishes on $U$ but not at $w$. Regarding $P$ as  a section of a line bundle we can suppose $\vert P(w)\vert=1$. Now $P$ also defines   holomorphic sections $\sigma_{i}$ of $L^{\lambda k}$ over $X_{i}$ for each $i$ (including $i=\infty$) which satisfy a fixed $L^{\infty}$ bound (because of the equivalence of the metrics on the line bundle). By construction the section $\sigma_{\infty}$ vanishes in a neighbourhood of $x_{0}$ and so by analytic continuation and the fact that the regular set is dense and connected it vanishes identically.  It follows  from the $L^{\infty}$ bound on $ \sigma_{i}$, the general estimate of (2.1) and convergence on compact subsets of the regular set that $\Vert \sigma_{i} \Vert_{L^{\infty}}$ tends to $0$ as $i\rightarrow \infty$. But this contradicts the fact that $\vert P(w)\vert =1$ (again using the equivalence of the two metrics on $L^{k}$).

(Notice that in this proof we do use the fact that $X_{\infty}^{\reg}$ has an analytic, not just $C^{2,\alpha}$, structure.)

Recall that we have compatible maps $T_{k}:X_{\infty}\rightarrow W_{k}$ and $f_{\lambda}:W_{\lambda k}\rightarrow W_{k}$. Proposition (4.8) implies that the $T_{k}$ asymptotically separate points, in the sense that the induced map from $X_{\infty}$ to $\lim_{\leftarrow} W_{k}$ is injective. What we want to show now is that in fact there is some fixed $k$ for which this is true.
\begin{lem}
We can find a $k$ so that all fibres of $T_{k}\rightarrow W_{k}$ are finite.
\end{lem}

First we can plainly use Proposition 4.6 to arrange that $T_{k}$ is generically 1-1, i.e. so that the fibre $T_{k}^{-1}(w)$ is a single point for a generic $w\in W_{k}$. As usual we may as well suppose that this happens for $k=1$ and hence for all $k$. Thus all maps $f_{\lambda}:W_{\lambda k}\rightarrow W_{k}$ are also generically $1-1$. 
 Our main theorem 1.1 and the first derivative estimate imply that there is a number $r>0$ so that for any $X\in {\mathcal K}(n,c,V)$ and any  point $x\in X$ there is a holomorphic section of $L$ which does not vanish on  the ball of radius $r$ about $x$. The argument extends easily to the limit space $X_{\infty}$ and $H^{0}(X_{\infty}, L)$.  Choose $k$ in accordance with Proposition 4.6 taking $\rho=r/2$ say. Thus if $p_{1}, p_{2}$ are two points in the same fibre $F=T_{k}^{-1}(w)$ of $T_{k}:X_{\infty}\rightarrow W_{k}$ the distance between them is less $r/2$. In other words the fibre $F$ is contained in the $r/2$ ball about $p_{1}$, so there is a section $s\in H^{0}(X_{\infty}, L)$ of $L$ which does  not vanish on $F$. By construction, $F$ maps by $T_{k\lambda}$ onto $F_{\lambda}=f_{\lambda}^{-1}(w)$ for any $f_{\lambda}: W_{\lambda k}\rightarrow W_{k}$. The section $s^{\lambda k}\in H^{0}(X_{\infty}, L^{\lambda k})$ defines one component of $T_{k\lambda}$ so the fact that $s$ does not vanish on $F$ implies that $F_{\lambda}$ lies in the corresponding affine subspace.   Since $F_{\lambda}$ is a compact algebraic set it must be finite. Thus all maps $f_{\lambda}: W_{k\lambda}\rightarrow W_{k}$ have finite fibres. Let $N(w)$ be the number of local irreducible components of $W_{k}$ at $w$.  Since $f_{\lambda}$ is generically 1-1 the number of points in $f_{\lambda}^{-1}(w)$ is at most $N(w)$. It follows then the number of points in $T_{k}^{-1}(w)$ is also finite, and in fact bounded by $N(w)$.

\begin{prop} 
We can find a $k$ so that $T_{k}$ is injective.
 \end{prop}

As usual we may as well suppose that the value of $k$ in the previous Lemma  is $1$. Thus
$T_{1}:X_{\infty}\rightarrow W_{1}\subset \bC\bP^{N_{1}}$ has finite fibres.
For any given point $w_{1}\in W_{1}$ we can find a $k$ such that $T_{1}^{-1}(w_{1})$ is mapped injectively to $W_{k}$ by $T_{k}$. It is clear then there is a decomposition of $W_{1}$ into a finite number of  quasi-projective subvarieties
$Z_{\alpha}$ such that $T_{1}^{-1}(Z_{\alpha})$ is a disjoint union of a  number $n_{\alpha}$  of copies of $Z_{\alpha}$.  Pick points $z_{\alpha}\in Z_{\alpha}$.
If for some $\alpha$ some $T_{k}$ separates the points $T_{1}^{-1}(z_{\alpha})$ then it is clear that $T_{k}$ separates  points in $T_{1}^{-1}(z)$ for generic $z\in Z_{\alpha}$. Now the Proposition follows from a simple induction argument, using induction on the maximal dimension of a $Z_{\alpha}$ with $n_{\alpha}>1$ and the number of components $Z_{\alpha}$ with this maximal dimension. 

We have now achieved our main goal---the central statement in Theorem 1.2. We have a continuous bijection $T_{k}:X_{\infty}\rightarrow W_{k}$ which is a homeomorphism,  since the spaces are  compact. As usual we may as well suppose that this $k$ is $1$, so all $T_{k}$ are homeomorphisms.

Recall that we denote the differential geometric singular set, the complement of $X_{\infty}^{\reg}$ by $\Sigma$. Let $S_{k}\subset W_{k}$ denote the algebro-geometric singular set.
\begin{lem} We can choose $k$ so that $T_{k}^{-1}$ maps $S_{k}$ to $\Sigma$. \end{lem}
Of course it is equivalent to say that $T_{k}$ maps $X_{\infty}^{\reg}$ to smooth points of $W_{k}$.
The proof is similar to that of the previous Lemma. It follows from Proposition 4.7 that for any given compact subset $K\subset X_{\infty}^{\reg}$ we can choose $k$ so that $T_{k}$ maps $K$ into the smooth points of $W_{k}$. On the other hand the singular set $S_{1}$ has  a finite number of irreducible components.
If there is a component which  meets $T_{1}(X_{\infty}^{\reg})$ we choose one of maximal dimension, say $V$. Thus there is a point $x\in X_{\infty}^{\reg}$ with $T_{1}(x)\in V$. We apply Proposition 4.7 with $K=\{x\}$ to find a $k$ such that $T_{k}(x)$ lies in the smooth set of $W_{k}$. Then it is clear that the number of irreducible components of $S_{k}$ is strictly less than for $S_{1}$, and the proof is completed by induction. 

As usual we can suppose that the $k$ in Lemma 4.11 is 1. In the next subsection we will show that, at least for K\"ahler-Einstein limits, the singular sets match up but we do not need to use this fact. 

 \begin{lem}
 We can choose a $k$ such that $W_{k}$ is a normal variety. 
\end{lem}

Suppose $W_{1}$ is not normal. Let $\nu:\hat{W}_{1}\rightarrow W_{1}$ be the normalisation. Thus $\nu$ is a bijection outside the singular set $S_{1}$ of $W_{1}$. It is a general fact that the pull back $\cL=\nu^{*}({\mathcal O}(1))$ is an ample line bundle on $\hat{W}_{1}$, so we can choose $k$ such that
sections of $\cL^{k}$ define a projective embedding of $\hat{W}_{1}$ in $\bP$ say.  The map $T_{1}: X_{\infty}^{\reg}\rightarrow W_{1}$ maps into the smooth part and so lifts to $\hat{T}_{1}:X_{\infty}^{\reg}\rightarrow \hat{W}_{1}$.  Clearly the pull back of $\cL$ to $X_{\infty}^{reg}$ by this map is identified with our polarising bundle $L$. Moreover, Theorem 1 implies that the metrics on the bundle agree up to a bounded factor. So the sections of $\cL^{k}$ over $\hat{W}_{1}$ define bounded sections of $L^{k}$ over $X_{\infty}^{reg}$ that is, elements of $H^{0}(X_{\infty}, L^{k})$. Write $U\subset H^{0}(X_{\infty}, L^{k})$ for the image of this map from $H^{0}(\hat{W}_{1}, \cL^{k})$. These sections define a map $\alpha$ from $X_{\infty}^{\reg}$ to $\bP$ and the definitions mean that this is just the composite of $\hat{T}$ with the above projective embedding of $\hat{W}_{1}$.  The subspace $U$ contains the kth. powers of sections in $H^{0}(X_{\infty}, L)$ which uniformly generate the fibres, so we have a first derivative estimate on the map $\alpha$. Hence $\alpha$ extends to a Lipschitz map, which we also call $\alpha$, from $X_{\infty}$ to $\bP$ with image $\hat{W}_{1}$.  Let $Z$ be the intersection of  smooth part of $\hat{W}_{1}$ with $\alpha(\Sigma)$. The Lipschitz bound implies that the Hausdorff dimension of $Z$ is at most $2n-4$ and it follows that any local holomorphic function defined on the complement of $Z$ extends holomorphically over $Z$ \cite{Sh}. This means that $H^{0}(X_{\infty}, L^{k})$ can be identified with bounded holomorphic sections of the hyperplane bundle over the smooth part of $\hat{W}_{1}$. But it is a basic general fact about a normal variety that its structure sheaf can be defined by  bounded holomorphic functions on the smooth part. So the subspace $U$ is in fact the whole of $H^{0}(X_{\infty}, L^{k})$. Thus $\alpha$ is exactly $T_{k}$ and $W_{k}$ is $\hat{W}_{1}$, and hence normal. 

To complete the story we have
\begin{lem} If $W_{1}$ is normal then $W_{k}$ is the embedding of $W_{1}$ defined by sections of ${\mathcal O}(k)$.
\end{lem}
This follows from the same argument as above.

\

We have now almost completed the proof of Theorem 1.2. For any given polarised limit space $X_{\infty}$ we can choose a $k$ so that $H^{0}(X_ {\infty},L^{k})$
represents $X_{\infty}$ as a normal variety and if $X_{i}$ is a sequence converging to $X_{i}$ in the Gromov-Hausdorff sense we can choose a convergent sequence of embeddings. (Notice that the only reason for passing to a subsequence in the statement of Theorem 1.2 is that we can have different polarisations on the same Riemannian limit space.) The last point is to show that there is a single $k_{1}$ which works for all $X_{\infty}$. But this follows from Gromov compactness and the easy fact that if $k$ has the desired property for $X_{\infty}$ it does also for all limit spaces sufficiently close to $X_{\infty}$, in the Gromov-Hausdorff sense.

To spell out a little more the consequences of Theorem 1.2, observe that now that we are considering embeddings the degree of $W$ is determined by $k_{1}$ and $V$. So (for theoretical purposes) we can operate in a fixed quasi projective Chow variety ${\mathcal T}$ parameterising normal $n$-dimensional subvarieties of the given degree in a suitable large projective space $\bC\bP^{N}$. \lq\lq Algebro-geometric convergence'' of $X_{i}$ to $X_{\infty}$ means convergence in ${\mathcal T}$. There is a universal variety ${\mathcal U}\rightarrow {\mathcal T}$ and by general facts (\cite{H}, Theorem 9.11) this is a flat family. So we see that if $X_{i}$ converge to $X_{\infty}$ in the Gromov-Hausdorff sense then   $X_{i}$ and $W=X_{\infty}$ can be realised as fibres in a flat family. So, for example, the Hilbert polynomials of $X_{i}$ and $ W=X_{\infty}$ are the same. 
\

There are  different ways of going about the proofs of Theorem 1.2.
We mention one elegant alternative, based on a result from the thesis of Chi Li \cite{CL}, Prop. 7. This in turn depends upon results of Siu  and Skoda. For $X$ in ${\mathcal K}(n,c, V)$ let $R_{X}$ be the graded ring
$$   R_{X}=\bigoplus_{k} H^{0}(X;L^{k}). $$
Then from standard theory we know that $R_{X}$ is finitely generated and
$X={\rm Proj}(R_{X})$. Assuming the lower bound in Theorem 1.1, Li proves an effective form of finite generation in the sense that if $\sigma_{i}$ is an  orthonormal basis in the finite dimensional space $\bigoplus_{j=0}^{(n+2) k_{0}} H^{0}(X, L^{j})$ then the $\sigma_{i}$ generate $R_{X}$ and for each $k$ there is a number $B_{k}$ such that any element of $L^{2}$ norm $1$ in $H^{0}(X, L^{k})$ can be expressed as a polynomial in the $\sigma_{i}$ with co-efficients bounded by $B_{k}$. It follows easily that for a polarised limit space $X_{\infty}$ the  graded ring $$  R_{X_{\infty}}= \bigoplus_{k} H^{0}(X_{\infty};L^{k})$$
is finitely generated. Then we can immediately define the algebraic variety $W$ as ${\rm Proj}(R_{X_{\infty}})$. Of course there is still some work to do in checking the properties of $W$.

\subsection{Further results}
We will now restrict attention to the case when $X_{\infty}$ is the limit of K\"ahler-Einstein manifolds $X_{i}\in {\mathcal K}(n,c,V)$ with Ricci curvature +1, -1/2 or 0. We suppose that $L=K_{X}^{-1}$ or $K_{X}^{2}$ in the first and second situations  and in the third situation we suppose that the manifolds are Calabi-Yau, so we have fixed holomorphic $n$ forms $\Theta_{i}$ over $X_{i}$ with $\Theta_{i}\wedge \overline{\Theta}_{i}$ the volume form. For brevity we just call this \lq\lq the K\"ahler-Einstein case''.

\begin{prop}
In the K\"ahler-Einstein case the map $T: X_\infty\rightarrow W$ takes the differential geometric limit singular set to the algebro-geometric singular set. 
\end{prop}

\begin{proof}
The argument in the previous subsection implies that $T$ maps the smooth set in $X_\infty$ to the regular set in $W$. So we need to show that if $T(p)$ is a smooth point of $W$, then the limit metric on $X_\infty$ is also smooth at $p$. Denote by $\omega_i$ the K\"ahler-Einstein metric  on $X_i$, and $\omega_i'$ the induced Fubini-Study metric. Then we have $\omega_i'=\omega_i+\sqrt{-1}\p\bp \phi_i$ with $\phi_i=k^{-1} \log \rho_k(\omega_i) $. By our main Theorem 1.1 and Proposition 2.1 there is a constant $C_1>0$ such that $|\phi_i|_{L^\infty}\leq C_1$  for all $i$. Also by arguments similar to the proof of  Lemma 4.3 we see that there is a constant $C_2>0$ such that   for all $i$ we have $|\nabla_{\omega_i}\phi_i|_{L^\infty}\leq C_2$, and  $\omega_i'\leq C_2\omega_i$.   Now write $Ric(\omega_i')=\lambda\omega_i'+\sqrt{-1}\p\bp h_i$, where $\lambda$ is $1$, $-\frac{1}{2}$ or $0$.  So with suitable normalization of $h_i$  we have the equation
  \begin{equation} \label{KE equation}
  \omega_i^n=e^{h_i+\lambda\phi_i}\omega_i'^n. 
  \end{equation}
Then it is not hard to see that  $\int_{X_i} h_i^2 \omega_i'^n\leq C_3$ for some constant $C_3>0$.  Now for any $p$ in $W^{reg}$, we choose a small neighborhood $B(p, \delta)\subset W^{reg}$. Then there are corresponding points $p_i\in X_i$, such that $B(p_i, \delta)$ converges smoothly to $B(p,\delta)$  in $\C\P^{N_k}$. By standard elliptic estimate we see that $|h_i|_{C^1(B(p_i,\delta/2), \omega_i')}$ is uniformly bounded. Then by (\ref{KE equation})  there is a $C_4>0$ such that $C_4^{-1}\omega_i\leq \omega_i' \leq C_4\omega_i$ in $B(p_i, \delta/2)$. Thus $|\nabla_{\omega_i'}\phi_i|_{L^\infty(B(p_i,\delta/2))}\leq C_4|\nabla_{\omega_i}\phi_i|_{L^\infty(B(p_i,\delta/2))}\leq C_4C_2 $. Then in $B(p_i,\delta/2)$ with respect to the metric $\omega_i'$, the right hand side of (\ref{KE equation}) has a uniform $C^1$ bound. Therefore we can apply the Evans-Krylov theory(see for example \cite{Bl})  to conclude that $|\phi_i|$ has a uniform $C^{2, \alpha}$  bound in $B(p_i, \delta/4)$. Then standard arguments show that all covariant derivatives of $\phi_i$(with respect to $\omega_i'$) are uniformly bounded, so the K\"ahler-Einstein metrics $\omega_i$ converge smoothly in a neighborhood of $p$.   \\

\end{proof}

\begin{prop}
In the K\"ahler-Einstein case, the algebro-geometric limit $W$ has log-terminal singularities. 
\end{prop}

\begin{proof}
By general theory, what the statement really means is that for any singular point $x$ in $W$, there is a neighborhood $U$, and a nowhere zero holomorphic $n$ form $\Theta$ on $W^{reg}\cap U$ with $\int_{ W^{reg}\cap U}\Theta\wedge \overline{\Theta}< \infty$. We first consider the cases $L=K_X^2$ and $K_X^{-1}$.  Previous discussion has shown that for any $x$,  there is  a neighborhood $U$ of $x$, an integer $k>0$, a constant $C>0$, and  a section $s$ of $L^k$ over $X_\infty\setminus \Sigma=W^{reg}$ with $C^{-1}\leq \Vert s(x)\Vert^2\leq C$ for $x\in W^{reg}\cap U$. Here the norm is taken with respect to the K\"ahler-Einstein metric.   When $L=K_X^2$,  we define $\Theta=(s\otimes \overline{s})^{\frac 1 {2k}}$, then 
 $$\int_{W^{reg}\cap U} \Theta\wedge \overline{\Theta}=\int_{W^{reg}\cap U} \Vert s\Vert^{\frac{1}{k}} dvol\leq C^{\frac {1}{2k}} Vol(W). $$
 When $L=-K_X$, we define $\Theta=(s^*\otimes \overline{s^*})^{\frac{1}{k}}$, where $s^*$ is the dual section of $s$. So $\Vert s^*\Vert=\Vert s\Vert ^{-1}$.  Then 
 $$\int_{W^{reg} \cap U} \Theta\wedge \overline{\Theta}=\int_{W^{reg}\cap U} \Vert s^*\Vert^{\frac{2}{k}} dvol\leq C^{-\frac{1}{k}}Vol(W). $$
 In the Calabi-Yau case,  since $\Theta_i$ has norm one, we easily see that there is a limit holomorphic volume form $\Theta$ on $X_\infty\setminus \Sigma =W^{reg}$  with norm one.  Then $\int_{W^{reg}} \Theta\wedge \overline{\Theta}=Vol(W). $
\end{proof}

\begin{rmk}
From the uniform bound of the K\"ahler potentials $\phi_i$, it is not hard to see that the K\"ahler forms $\omega_i$ converge to a singular K\"ahler-Einstein metric $\omega_\infty$ on $W$ in the sense of \cite{EGZ}.
\end{rmk}

\section{Structure of three dimensional tangent cones}

In this section we make a more detailed study of the structure of the tangent cones occurring in the previous sections, in particular we prove

 \begin{thm}\label{thm-orbifold}
In complex dimension three, the link $Y$ of any tangent cone of the Gromov-Hausdorff limit $X_\infty$  is a five dimensional Sasaki-Einstein orbifold.
\end{thm}

As before we write $Y=Y^{reg}\cup \Sigma$, where $Y^{reg}$ is the smooth part and $\Sigma$ the singular part. We view $Y$ as the radius one link in $C(Y)$. For any $q\in \Sigma$, any tangent cone of $C(Y)$ at $q$ splits at least one line, so by general theory(see for example \cite{Ch2}) must have the form $\C\times\C^2/\Gamma$ for some $\Gamma\in U(2)$. Moreover, $\Gamma$ depends only on $q$.  So  the tangent cone of $Y$ at $q$ is $\R\times \C^2/\Gamma$. 

\begin{lem}
 $Y^{reg}$ is geodesically convex in $Y$. 
\end{lem}

\begin{proof}
 For any two points $p$, $q$ in $Y$, it is a general fact that a minimizing geodesic in $C(Y)$ connecting $p$ and $q$ must be of the form $(r(t), \gamma(t))$ where $\gamma(t)$ is a geodesic in $Y$, and $r$ is a universal function of $d_Y(p, q)$ and $t$ determined by elementary trigonometry. By recent result of Colding-Naber \cite{CN} we know $C(Y^{reg})$ is geodesically convex in $C(Y)$, so the lemma follows.
 \end{proof}

As usual there is a Reeb field $\xi=Jr\frac{\p}{\p r}$ on $C(Y^{reg})$, which is holomorphic, Killing, of unit length, and tangent to $Y^{reg}$. For any $p\in Y^{reg}$ we denote by $p(t)$ the integral curve $\exp(t\xi).p$. For $|t|$ sufficiently small, $p(t)$ defines a geodesic segment in $Y^{reg}$. 

\begin{lem}
For any  $p_1, p_2\in Y^{reg}$, if $p_1(t), p_2(t)$ are both defined on some interval $[0, T]$, then $f(t)=d(p_1(t), p_2(t))$ is independent of $t$.
\end{lem}
\begin{proof}
By Lemma 5.2 for any $t\in [0, T]$, the minimizing geodesic $\gamma$ connecting $p_1(t)$ and $p_2(t)$ lies in $Y^{reg}$. So there is an $\epsilon>0$ so that the curve $\gamma_s=\exp(s\xi). \gamma$ is in $Y^{reg}$ for $s\in[0, \epsilon]$. Clearly the length of $\gamma_s$ is independent of $s$. Thus $f(t)$ is a decreasing function. Replace $\xi$ by $-\xi$ one sees that $f$ is also an increasing function. Thus $f$ is constant. \\
\end{proof}

\begin{prop}
$\xi$ generates a one parameter group of isometric actions on $Y$. 
\end{prop}
\begin{proof}
 Fix any point $p$ in $Y^{reg}$,  choose a convex embedded ball $B(p, r)$ in $Y^{reg}$. 
We claim $B(p(t), r)\cap \Sigma=\emptyset$ for all $t$. For otherwise there is a $T>0$ such that $B(p(t), r)\cap \Sigma=\emptyset$ for $t\in [0, T]$ but $\p B(p(T), r)\cap \Sigma$ is non-empty. Choose a point $q$ in this intersection.   Let $\gamma:[0, 1]$ be the radial geodesic connecting $p(T)$ and $q$, and let $p_i=\gamma(1-2^{-i})$. Then $B_i=B(p_i, 2^{-i}r)\subset B(p(T), r)$, and $d(p_i, q)=2^{-i}$. Consider the pointed sequence $(Y, 2^{i}d_Y, q)$. By assumption we know  as $i$ tends to infinity by passing to a subsequence this converges to  a tangent cone $Y_q=\R\times \C^2/\Gamma$. Then  the rescaled balls $2^{i}B_i$ converge to a ball  $B(p_{\infty}, r)$ in $Y_q$ and $d(p_{\infty}, 0)=r$. But $B_i$ is isometric to a ball in $B(p, r)$ so have uniformly bounded geometry and thus $2^{i}B_i$ converges to a flat ball $B_\infty$. Moreover by Lemma 5.3  the distance between any two points in $B_\infty$ is realized by the length of a geodesic within $B_\infty$. Clearly this can not happen on $Y_q$. 
 
By the claim the isometric action $\exp(t\xi)$ is well defined on $Y^{reg}$ for all $t$. Then we can extend the action  to an isometric action on $Y$: given $p\in \Sigma$ we pick a Cauchy sequence $p_i\in Y^{reg}$ converging to $q$; for any $t$, $p_i(t)=\exp(t\xi)$ is also a Cauchy sequence in $Y^{reg}$, so there is a unique limit $p(t)$. We define $exp(t\xi). p=p(t)$. Clearly $\exp(t\xi)$ is  distance preserving. Moreover $\exp(t\xi)$ preserves both $Y^{reg}$ and $\Sigma$. 
\end{proof}

We denote by $\psi(t)=\exp(t\xi)(t\in \R)$ the above one parameter group action. Then we have

\begin{lem}
There is no point in $Y$ fixed by $\psi$. 
\end{lem}
\begin{proof}
If $q$ is a fixed point, then clearly $q\in \Sigma$. Choose a tangent cone $Y_q=\R\times \C^2/\Gamma$ at $q$. The action of $\psi$ induces a one parameter group of isometric actions on $Y_q$, which fixes the origin. On the other hand on the smooth part of $Y_q$ the  corresponding infinitesimal action is given by a Killing field  of  constant length. Clearly such a Killing field can not have zeroes. Contradiction.
\end{proof}

 Now we are ready to conclude that

\begin{prop}
$\Sigma$ is a disjoint union of finite many periodic orbits of $\psi$.
\end{prop}
\begin{proof}
Fix any $q\in\Sigma$. Since it is not a fixed point of $\psi$, we can choose a neighborhood $B_r(q)$ such that any path-connected component of the intersection of an orbit of $\psi$ with $\overline{B_r(q)}$ is compact.  Let  $O_q$ be one path-connected component of $\psi(q)$ in $B_r(q)$. We claim for $s>0$ sufficiently small, $\Sigma\cap B_s(q)=O_q\cap B_s(q)$.  If not, then there is a sequence $p_i\in (B_r(q)\setminus O_q)\cap \Sigma$ converging to $q$. We can choose $q_i$ on the path-connected component of the orbit of $p_i$ in $\overline{B_r(q)}$ which has least distance to $q$. Then $s_i=d(q, q_i)>0$. For $i$ sufficiently large we have $d(q_i, O_q)=s_i$.  Now consider the rescaled pointed sequence $(B_r(q), s_i^{-1}d_Y, q)$. As $i\rightarrow\infty$, by passing to a subsequence, this converges to $\R\times \C^2/\Gamma$. Moreover, $O_q$ converges to $\R\times \{0\}$, and $q_i$ converges to $q_\infty$ which has distance $1$ to $\R\times \{0\}$. But $q_i$ is singular for all $i$, so $q_\infty$ is also singular. Contradiction.  Then the Proposition follows from the claim and an obvious compactness argument. 
 \end{proof}
  
  Now we pick a point $q$ in $\Sigma$. Choose a neighborhood $U_q$ of  $q$ such that $\Sigma\cap U_q=O_q$ consists of exactly one component.  Then one can take a local quotient of $U_q$ by  $\psi$, and obtain a four dimensional (incomplete)  metric ball $B(q,200)$(say radius is $200$) with an isolated singularity $q$. Moreover,  the tangent cones at $q$ are all isometric to $\C^2/\Gamma$ for a unique $\Gamma\subset U(2)$. The metric $g$ on the smooth part $B(q,200)\setminus\{q\}$ is K\"ahler-Einstein. 
  We write $B=B(q, 100)$, and $B^*=B(q,100)\setminus\{q\}$. Denote by $\hat B$ the standard ball of radius $100$ in $\C^2/\Gamma$, and $\hat{B}^*=B\setminus \{0\}$.
   
   \begin{thm}\label{thm-ABT}
   There is a diffeomorphism $F: \hat B^*\rightarrow B^*$ such that $F^*g$ extends to  a smooth orbifold Riemannian metric on $\hat B$. 
   \end{thm}
   
   Given this theorem then it is not hard to prove Theorem \ref{thm-orbifold}. So on the local quotient $B$ we have an orbifold chart $\{z^i\}$ with K\"ahler metric $\omega=\sqrt{-1}\p\bp \phi$.  We pull back the coordinate $\{z^i\}$ to $U_q$. Let $\eta$ be the contact form associated to the Sasaki structure on the smooth part $U_q^0=U_q\setminus O_q$. Then the $1$-form $\eta'=\eta-2Im(\p_z \phi)$ is closed. Clearly $H^1(U_q^0, \R)=0$, so $\eta'=dx$ for some function $x$.  Then it is easy to see that $\xi=\frac \p {\p x}$, in the coordinate $(x, z^1,  z^2)$. This gives rise to an orbifold chart for $U_q$. The compatibility condition between the orbifold charts follows easily from the local action $\psi$. \\

  Theorem \ref{thm-ABT} is certainly well-known, due to Anderson \cite{An}, Bando-Kasue-Nakajima \cite{BKN}, and Tian \cite{Tian1}.  We include a proof here for the convenience of readers. 
    For simplicity of notation we assume $\Gamma$ is trivial, and the proof is the same for a general $\Gamma$. For any $a_1<a_2$, we denote $A(a_1, a_2)=\{p\in B|a_1<d(p, q)<a_2\}$ and $\hat A(a_1,a_2)=\{x\in \C^2|a_1< |x|< a_2 \}$. 
    Since any tangent cone at $p$ is isometric to $\C^2/\Gamma$, by general results of Anderson, Colding,  there is a $\delta\in (0, \frac{1}{10})$ such that for $r$ sufficient small there is an embedding $\phi_r: \hat A(1-\delta, 100+\delta)\rightarrow B(q, 200)$ such that $(1-\epsilon(r))|x|\leq r^{-1}d(q, \phi_r(x))  \leq (1+\epsilon(r))|x|$ and $|r^{-2}\phi_r^*g-g_0|_{C^4}\leq \epsilon(r)$, where  $\epsilon(r)$ is a monotone function that goes to $0$ as $r$ tends to $0$. Here and from now on, the norm of a quantity defined on an annulus in $\R^4$ is always taken with respect to the Euclidean metric.  Then we readily see that for all $r<s<1$, there is a deformation retract from $A(r, 1)$ to $A(s, 1)$, and $B$ is homeomorphic to $\hat B$.  The proof of Theorem \ref{thm-ABT} is divided into four steps:\\
    
 \textbf{Step I}($C^0$ chart):\\
    
To construct a chart so that $g$ is continuous we need to glue together the above almost Euclidean  annuli in a controllable way.  This is elementary and we begin with the following lemma

\begin{lem}\label{gauge-fixing}
For $\epsilon>0$ sufficiently small, there is a constant $K(\epsilon)>0$ which goes to zero as $\epsilon$ tends to zero,  such that for any smooth map $\phi: \hat A(30,80) \rightarrow \R^4$ with $|\phi^*g_0-g_0|_{C^4(\hat{A}(30,80))}\leq \epsilon$,  there is an isometry $P$ of $\R^4$ such that $|P\circ \phi-Id|_{C^3(\hat A(40, 70))}\leq K(\epsilon)$.
\end{lem}
\begin{proof}
Assume the statement fails, then there is a constant $\tau>0$,  a sequence $\epsilon_i\rightarrow0$, and maps $\phi_i: \hat A(30,80) \rightarrow \R^4$ with 
$|\phi_i^*g_0-g_0|_{C^4(\hat{A}(30,80))}\leq \epsilon_i$, but for any isometry $P$ we have $|P\circ \phi_i-Id|_{C^3(\hat A(40, 70))}\geq \tau$. Then $\phi_i$ converges to a map $\phi_\infty$ in $C^3(\hat A(40, 70))$, such that $\phi_\infty^*g_0=g_0$. So $\phi_\infty$ is an isometry of $\R^4$.  Since $|\phi_\infty^{-1}\circ \phi_i-Id|_{C^3(\hat A(40, 70))}$ converges to zero as $i$ goes to infinity. We arrive at a contradiction.
 \end{proof}
 
     \begin{lem}\label{matching}
Suppose two maps $f_0: \hat A(1, 100)\rightarrow B(q, 200)$,  $f_1: \hat A(1-\delta, 100+\delta)\rightarrow B(q, 200)$ satisfy that for $i=0, 1$ and some $r>0$,
 $(1-\epsilon)|x|\leq 10^i r^{-1}d(q, f_i(x))  \leq (1+\epsilon)|x|, $ and  $|10^{2i}r^{-2}f_i^*g-g_0|_{C^4}\leq \epsilon$ on $\hat A(10^i, 10^{i+1})$.
Then there is a constant $G=G(\epsilon)$ with $\lim_{\epsilon\rightarrow 0}G(\epsilon)=0$, a rotation $R\in O(4)$, and  a map $f: \hat A(10^{-1}, 100)\rightarrow B(q,200)$, with 
$f(x)=f_0(x)$ on $\hat A(9,100)$, $f(x)=f_1(10R^{-1}(x))$ on $\hat A(10^{-1}, 2)$, and  $|r^{-2}f^*g-g_0|_{C^2}\leq  C(\epsilon)$ on $\hat A(10^{-1}, 100)$.
  \end{lem}

\begin{proof} By the obvious scaling invariance we may assume $r=1$.
Let $D=Im(f_0)\cap Im(f_1)$. Since $\epsilon$ is small, we may assume $\hat A(3, 8)$ is contained in $f_0^{-1}(D)$. Then there is a constant $C_1$ independent of $\epsilon$ such that the map $\psi=10^{-1}f_1^{-1}\circ f_0: \hat A(3, 8)\rightarrow \R^4$ satisfies $|\psi^*g_0-g_0|_{C^4}\leq C_1\epsilon$, and $(1-3\epsilon) |x|\leq|\psi(x)|\leq (1+3\epsilon) |x|$.  By Lemma \ref{gauge-fixing}  there is an isometry $P$ of $\R^4$ such that $|P\circ \psi-Id|_{C^3}\leq K(C_1\epsilon)$ on $\hat A(4, 7)$. We write $P(x)=R(x+\xi)$ for a rotation $R$ and a translation $\xi$. Then it is easy to see  that $|R\circ \psi-Id|_{C^3}\leq C_2(\epsilon)$ with $\lim_{\epsilon\rightarrow0}C_2(\epsilon)=0$, and $R(\hat A(1-\delta, 100+\delta))$ contains $\hat A(1, 100)$.    Choose a cut-off function $\chi(x)$ on $\hat A(1, 100)$ with  $\chi(x)=1$ for  $|x|\leq 5$ and $\chi(x)=0$ for $|x|\geq 6$.  Using the map $f_0$ we get a corresponding cut-off function on $B(q, 200)$, still denoted by $\chi$. Then $|\chi|_{C^4_g}\leq C_3$ for a constant $C_3$ independent of $\epsilon$.  Clearly $\chi(p)=0$ when $p \notin Im(f_1)$ and $\chi(p)=1$ when $p\notin Im(f_0)$.  Define 
$h: Im f_0\cup Im f_1\rightarrow\R^4$  sending $p$ to $10^{-1}\chi(p) R\circ f_1^{-1}(x)+(1-\chi(p))f_0^{-1}(x)$.  Then for $\epsilon$ sufficiently small we have $h=f_0^{-1}$ on $A(8, 100)$ and $h=10^{-1}R\circ f_1^{-1}$ on $A(10^{-1}, 3)$, and $|h^*g_0-g|_{C^2_g}\leq C(\epsilon)$ with $\lim_{\epsilon\rightarrow0}C(\epsilon)=0$.  Define $f=h^{-1}$. Then 
$f(x)$ meets the required properties.
\end{proof}

 \begin{prop}\label{C0}
There is a diffeomorphism $F: \hat B^*\rightarrow B^*$ such that $F^*g$ extends to a $C^0$ metric tensor over $B$.  
\end{prop}

\begin{proof}

Since the problem is local, we may assume  for all $r\leq 1$ that the above map $\phi_r$ exists and $\epsilon(1)$ is as small as we like. For simplicity we denote $\phi_k=\phi_{10^{-k}}$, and $\epsilon_k=\epsilon(10^{-k})$.  Now we first define $F_0(x)=\phi_0(x)$ on $\hat A(1, 100)$. Inductively suppose $F_k$ is defined on $\hat A(10^{-k}, 10^{-k+2})$ satisfying $F_k(x)=\phi_k\circ R_k^{-1}(10^k x)$ on $\hat A(10^{-k}, 20\cdot 10^{-k})$ for some rotation $R_k\in O(4)$, then we apply Lemma \ref{matching} to the two maps $\phi_k\circ R_k^{-1}$ and $\phi_{k+1}$ with $r=10^{-k}$ and $\epsilon=\max(\epsilon_{k-1}, \epsilon_k)$,  and obtain a map $f_{k+1}$ defined on $\hat A(1/10, 100)$ satisfying (2). Then we define $F_{k+1}(x)$ to be $f_{k+1}(10^kx)$ on $\hat A(10^{-k-1}, 10^{-k+1})$. By Lemma \ref{matching} we see that all the $F_k$'s match together to a map $F$ from $\hat B^*$ to $B(q, 200)$, and we can modify $F$ slightly near $\p \hat B$ so that the image is exactly $B^*$. It is easy to see that
$|F^*g-g_0|_{L^\infty(\hat A(10^{-k}, 10^{-k+1}))}=|10^{2k}f_{k+1}^*g-g_0|_{L^\infty(\hat A(10,100))}\leq G(\max(\epsilon_{k-1}, \epsilon_{k}))$, and $F^*g$ extends to a continuous metric tensor over $B$. 
\end{proof}

\textbf{Step II}(Curvature bound):\\

Now  we may assume $g$ is a $C^0$ metric on $B=\hat{B}$.  

 \begin{lem}
 We have
 $$\int_{B^*} |Rm(g)|^2 dvol_g<\infty.$$
 \end{lem}
 \begin{proof}
Let  $A_{\pm}$ be the connection induced by the Levi-Civita connection of $g$ on $\Lambda^{\pm}_g$. The Einstein condition implies $A_+$ is self-dual and $A_{-}$ anti-self-dual with respect to $g$.  Thus $$|Rm(g)|^2dvol_g=Tr(F_{A_+}\wedge F_{A_+}-F_{A_-}\wedge F_{A_-}). $$ By the tangent cone condition we can easily find a smooth family of spheres $S_r$ in $B^*$ with the property that as $r$ tends to zero, $(S_r, r^{-2}g)$ converges smoothly to the round sphere in $\R^4$, and the restriction to $(S_r, r^{-2}g)$ of the connection $A_{\pm}$ converges to the trivial flat connection. Then for any $s<r$
 $$\int_{A(s,r)}TrF_{A_+}\wedge F_{A_+}=CS(A_+, S_r)-CS(A_+, S_{s})(mod \ \Z), $$
 where $CS(A, M)=\int_M dA\wedge A+\frac{2}{3}A\wedge A\wedge A$ is the Chern-Simons invariant of a connection $A$ over a three manifold $M$, defined modulo $\Z$. 
By assumption, $CS(A_+, S_r)=CS(A_+, \frac{1}{r}S_r)\rightarrow 0$ as $r\rightarrow 0$. So we choose $r$ small enough so that for any $s\leq r$ we have $|CS(A_+, S_r)|\leq 1/8$ modulo $\Z$. So $\int_{A(s,r)}TrF_{A_+}\wedge F_{A_+}$ is in $[-1/4, 1/4]$ modulo $\Z$, and on the other hand it clearly depends continuously on $s$, so the integral is uniformly bounded for all $s<r$. One can similarly deal with $A_-$. Together this implies $\int_{B^*}|Rm(g)|^2dvol_g$ is finite. 

\end{proof}
 
 \begin{prop} \label{curvature bound}
 For any $k\geq 0$, $|\nabla ^k_gRm(g)|$ is  uniformly bounded in $B^*$.
 \end{prop}
 \begin{proof} Since the metric $g$ is $C^0$ equivalent to the flat metric $g_0$, the Sobolev space $W^{1,p}$ is the same with respect to both metrics, and the Moser iteration works for the operator $\Delta=\Delta_g$. Here again we use the geometers' convention for the sign. By Bochner formula  there is a constant $C_1>0$ such that 
 $$\Delta|Rm|\leq C_1|Rm|^2, $$
 which is on the borderline of applying Moser iteration.  Due to Bando-Kasue-Nakajima \cite{BKN} (Corollary 4.10),  there is an improved Kato's inquality, namely, there are $C_2>0$ and $\delta\in (0,1)$, such that 
 $$\Delta|Rm|^{1-\delta}\leq C_2|Rm|^{2-\delta}. $$
 Let $u=|Rm|^{1-\delta}$ and $f=|Rm|$. Then we can apply  \cite{Sib}(Lemma 2.1) with $q=\frac{1}{1-\delta}$ and $q_0=\frac{1}{2(1-\delta)}$
  to conclude that 
 $|Rm|$ is in $W^{1,2}$. By Sobolev embedding we see $|Rm|\in L^4$. Also that $|\nabla Rm|\in L^2$ implies that the inequality $\Delta |Rm|\leq C_1|Rm|^2$ holds weakly on the whole ball $B$. Then we can apply the standard Moser iteration to conclude $|Rm|$ is uniformly bounded. Now consider $|\nabla Rm|$. For any $p\in B^*$ with $d(p, q)=r\leq  1/2$, the rescaled ball $r^{-1}B(p, r/2)$ has uniformly bounded geometry, so standard elliptic regularity for the Einstein equation then implies that 
 $|\nabla Rm|\leq C_3 r^{-1}. $ for some constant $C_3>0$.
 Thus $|\nabla Rm|\in L^3$. By Bochner formula again there is a constant $C_4>0$ such that  $$\Delta|\nabla Rm|\leq C_4|Rm||\nabla Rm|. $$
 Let $u=|\nabla Rm|$, $f=C_4|Rm|$ and apply  \cite{Sib}(Lemma 2.1)  with $q=1$, and $q_0=3/4$, we get $|\nabla Rm|\in W^{1,2}$. Thus the inequality holds weakly on $B$ and by Moser iteration $|\nabla Rm|$ is  uniformly bounded.  Then similarly one can prove the bound for higher covariant derivatives of the curvature tensor. 
 \end{proof}
 
 \textbf{Step III}($C^{1, \alpha}$ chart):\\
 
 To construct a coordinate chart so that $g$ is $C^{1, \alpha}$, we shall use Rauch comparison theorem, following \cite{BKN}. The following lemma is a direct consequence of the tangent cone condition(by using the maps $\phi_r$):
 
 \begin{lem}\label{approximate sphere}
 There is a sequence $\epsilon_i\rightarrow 0$ and a sequence of smooth embeddings $f_i$ from $S^3$ to $B$ with the properties
\begin{enumerate}
\item $d_{GH}(S_i, \p B(i^{-1}))\leq i^{-1}\epsilon_i$ where $S_i=f_i(S^3)$. 
\item $|i^{2}f_i^*g-h_0|_{C^4_{h_0}}\leq \epsilon_i$, where $h_0$ is the standard round metric on $S^3$. 
\item $|i^{-1}A_{S_i}+Id|_{C^3_{h_0}}\leq \epsilon_i,  $ where $A_{S_i}: TS_i\rightarrow TS_i$ is the shape operator.
\end{enumerate}
\end{lem}

\begin{prop}
There is a $C^3$ diffeomorphism $F:  B^* \rightarrow B^*$ such that $F^*g$ extends to a $C^{1,1}$ metric tensor on $B$. 
\end{prop}

\begin{proof} 
We define $F_i: S^3\times [i^{-1}, 1]$, sending $(x, t)$ to $\exp_{f_i(x)}((t-i^{-1})N(x))$, where $N(x)$ is the outward normal vector at $x$.  Consider a Jacobi field $J(t)$ along a geodesic $\gamma_x(t)=F_i(x, t)$. Then since the curvature of $g$ is uniformly bounded,  by Rauch comparison theorem there are constants $C_1>0$ and $\delta>0$ independent of $x$ and $i$ such that $C_1^{-1}|J(i^{-1})|_g\leq |J(t)|_g \leq C_1i|J(i^{-1})|_g$ for $t\in[i^{-1}, \delta]$. For simplicity of notation we may assume $\delta=1$. So for $i$ large enough $F_i$ has no critical points in $[i^{-1}, 1]$. Indeed $F_i$ is a diffeomorphism. For otherwise  there would be a geodesic loop $\sigma(s)(s\in [0, T])$ which is perpendicular to $S_i$ when $s=0$ and $s=T$. It is then easy to see this could not happen for sufficiently large $i$, by passing to a tangent cone. \\

Now we write $F_i^*g=dt^2+t^2h_i(t)$. First we notice that  $|d_g(0, F_i(x,t))-t |\leq i^{-1}\epsilon_i$.  Now we derive estimates for $g_i(t)$. 
Given a unit tangent vector $\xi$ at $x\in S^3$.  Let $J(t)$ be the Jacobi field along $\gamma_x(t)$ with $J(i^{-1})=df_i(\xi)$ and $\dot{J}(i^{-1})=A_{S_i}(J(i^{-1}))$. Then 
$J(t)=d{F_i}_{(x, t)}(\xi)$. Clearly $||J(i^{-1})|_g-i^{-1}|\leq i^{-1}\epsilon_i$ and  $|\dot{J}(i^{-1})-iJ(i^{-1})|_g\leq 2\epsilon_i$. Let $\{e_1(t), \cdots, e_n(t)=\dot{\gamma}_x(t)\}$ be an
orthonormal frame of parallel vector fields along $\gamma_x(t)$, such that $J(i^{-1})=|J(i^{-1})|_ge_1$. Under the decomposition $J(t)=\sum _{\alpha}J_\alpha(t) e_\alpha(t)$ we have
$$\ddot{J}_\alpha(t)+\sum_\beta R_{\alpha n\beta n}(\gamma_x(t)) J_\beta(t)=0, $$
where $R_{\alpha n \beta n}=R(e_\alpha, e_n, e_\beta, e_n)$.  From the above discussion we have $|J(t)|\leq 2C_1$ for $t\in [i^{-1}, 1]$. 
So it is easy to see that there is  a constant $C_2>0$ such that  $$||J(t)|_g-t|\leq C_2(i^{-1}+\epsilon_i t+t^3). $$ Thus
$$|h_i(t)-h_0|_{L^\infty_{h_0}}\leq C_2(i^{-1}t^{-1}+\epsilon_i +t^2). $$
Now take a unit tangent vector $X$ at $x$, we vary $J(i^{-1})$ so that  $\nabla^0_X J(i^{-1})=0$ at $x$, and extend $X$ to a unit tangent vector field in a neighborhood $U$ of $x$ in $S^3$. We may also view $X$ as a tangent vector field on $U\times [i^{-1}, 1]$. Now we differentiate the Jacobi field equation, and  similar arguments as above yield
$$|\nabla_X J(t)|_g\leq  C_3(i^{-1}+\epsilon_i t+t^3),  $$
for a constant $C_3>0$. 
This implies that  there is a constant $C_4>0$ such that $t^{-1}|\nabla^0(h_i(t)-h_0)|_{h_0}\leq C_4(i^{-1}t^{-2}+\epsilon_it^{-1}+t).  $ Similarly one can get bounds on higher derivatives of $h_i(t)-h_0$. The point is that for a fixed $\tau>0$ as $i$ goes to infinity we know $F_i(x,t)$  converges in $C^3$ to a limit $F^\tau_\infty(x, t)$ on $S^3\times [\tau, 1]$. Then we can let $\tau\rightarrow0$  and obtain a limit $F: S^3\times (0, 1]$ with the property that $d_g(0, F(x,t))=t$, and  
$$t^{-2}|F^*g-g_0|_{C^0_{g_0}}+t^{-1}|F^*g-g_0|_{C^1_{g_0}}+|F^*g-g_0|_{C^2_{g_0}}\leq C_5 , $$ 
for some constant $C_5>0$. This implies that $F^*g$ extends to a $C^{1,\alpha}$ metric on $B$. 
\end{proof}

\textbf{Step IV}($C^\infty$ chart):\\

Now we may assume $g$ is a $C^{1,\alpha}$ metric on $B$.  Notice the metric $g$ is also K\"ahler, and compatible almost complex structure $J$ is $C^{1, \alpha}$ in $B$. Thus by the integrability theorem \cite{NW}, modifying by a  $C^{2, \alpha'}(\alpha'<\alpha)$ diffeomorphism, we may assume $J$ is the standard complex structure near the origin. So in a small ball $B_\epsilon$ the K\"ahler form of $g$ is of the form $\omega=\sqrt{-1}\p\bp \phi$ for a real valued function $\phi$ with regularity $C^{3, \alpha'}$. The K\"ahler-Einstein  equation on $B_\epsilon^*$ has the form
$$(\sqrt{-1}\p\bp \phi)^2=e^{-\phi+h}\omega_0^2, $$
where $h$ is a pluri-harmonic function on $B_\epsilon^*$ and $\omega_0$ is the standard K\"ahler form on $\C^2$. By Hartogs theorem $h$ extends smoothly to $B_\epsilon$. Then the standard elliptic regularity implies that $\phi$ and hence $g$ is smooth on $B_\epsilon$.  This finishes the proof of Theorem \ref{thm-ABT}. 

\subsection{Further discussion}

We can use this detailed description of the link $Y$, in the three-dimensional case to get a more precise understanding of the \lq\lq topological obstruction''
of Section 3.2.2. A representation $\alpha:\pi_{1}(Y\setminus \Sigma)\rightarrow S^{1}$ defines a covering of $Y\setminus \Sigma$ and it is clear that the metric completion of this is again an orbifold $\tilde{Y}$ with a metric of Ricci curvature $(2n-1)$. It is clear then that the usual proof of Myers Theorem extends to show that $\tilde{Y}$ is compact, so the representation maps to a finite group. Thus $\pi_{1}(Y\setminus \Sigma)$ is also finite and the torus $T$ in the discussion of 3.2.2 is in this case trivial. (Of course the set $Y_{\epsilon}$ can be assumed to be homotopy equivalent to $Y\setminus \Sigma$).  Moreover it is also clear that the usual proof of the Bishop Theorem extends to this case to show that the volume of $\tilde{Y}$ cannot exceed that of $S^{2n-1}$.
Hence the order of the cover, is bounded by $\kappa^{-1}$ where $\kappa$ is the volume ratio, and hence by $c^{-1}$. Let $D=D(c)$ be the least integer such that all integers less than or equal to $c^{-1}$ divide $D$.  Then we see that the power $\alpha^{D}$ of any such  representation  must be trivial. Thus if, from the beginning of the discussion in Section 3, we consider powers $L^{Dk}$ we never encounter the topological obstruction. The point here of course is that $D$ is determined in a simple explicit way by $c$ which in turn, in the Fano case, is known explicitly. In many practical cases of interest $D$ is not too large. 

We expect that in fact the same will be true in higher dimensions (with the same $D(c)$). Of course we do not expect that the singularities will always be of orbifold type, but it seems likely that the Bishop theorem can still be extended to the metric completion of a covering, as above. There is a slightly weaker statement which should be easier to prove. Let $y$ be a point in the singular set $\Sigma_{Y}$ of a $(2n-1)$-dimensional link $Y$. Let $B$ be a sufficiently small ball about $y$ and $B^{\reg}\subset B $ the regular set. Suppose that we have found a number $E$ such that for all such points
(in all tangent cones of all limits of manifolds in ${\mathcal K}(n,c,V)$) the homology group
$H_{1}(B^{\reg}, \bZ)$ has order bounded by $E$. Let $\alpha$ be a representation of $\pi_{1}(Y\setminus \Sigma)$ as above. Then in the covering defined by $\alpha^{E}$ the pre-image of $B^{\reg}$ is a disjoint union of copies of $B^{\reg}$. In this situation it is straightforward to apply recent results of Colding and Naber \cite{CN} to show that the regular set in the metric completion $\tilde{Y}$ is geodesically convex, and then to extend the Bishop argument to this case. Then we see that  if, from the beginning of the discussion in Section 3, we consider powers $L^{DE k}$ then we never encounter the topological obstruction. Arguing by induction on dimension it seems likely that in fact the number $E= D^{n-2}$ will have the property stated above so, for this weaker statement, we would consider powers $L^{D^{n-1} k}$. But, in fact it seems to us most likely that these higher powers of $D$ are not  required.

 In this direction we make the following conjecture, which (if true) would be a substantial sharpening of Theorem 1.1.
\begin{conj}
For any $n,c,V$ and $\eta<1$ there is a number $k_{0}(n,c,V,\eta)$ such that if $k\geq k_{0}$ then for any $X$ in ${\mathcal K}(n,c,V)$ we have
             $$  \eta (2\pi)^{-n} (kD)^{n}\leq  \rho_{kD,X}\leq \eta^{-1} c^{-1}(2\pi)^{-n} (kD)^{n} , $$
with $D=D(c)$ as above.
\end{conj}
To put this in context, recall that for a {\it fixed} $X$ the standard asymptotics is $\rho_{k,X}\sim (2\pi)^{-n} k^{n}$ as $k\rightarrow\infty$. This essentially follows from the fact that on $\bC^{n}$ we have $\rho= (2\pi)^{-n}$. The conjectural lower bound here is a uniform version of this over ${\mathcal K}(n,c,V)$, provided we work over multiples of $D$. On the other hand the corresponding upper bound---$\rho_{kD,X}\leq \eta^{-1}(2\pi)^{-n} (kD)^{n}$---almost certainly fails, because at the vertex of a cone $C(Y)$ we have  $\rho= \kappa^{-1} (2\pi)^{-n}$ where $\kappa\geq c$ is the volume ratio. This is why we believe that the  plausible  upper bound should include the extra factor $c^{-1}$. In a similar way, if in fact we do encounter the topological obstruction of 3.2.2 in some limit space, then it seems it would not be true that there is a lower bound on $\rho_{k,X}$ for {\it all} sufficiently large $k$, since the twisting of the line bundle will force $\rho$ to be small as we approach the singularity.
This phenomenon---that near to a singularity $\rho$ gets larger or smaller  depending on divisibility---is similar to the orbifold situation considered by Ross and Thomas in \cite{RT}.


\begin{thebibliography}{s}
\bibitem{An} M. Anderson. \emph{Ricci curvature bounds and Einstein metrics on compact manifolds.} J. Amer. Math. Soc. 2 (1989), no. 3, 455-490.
\bibitem{BKN} S. Bando, A. Kasue, H. Nakajima. \emph{
On a construction of coordinates at infinity on manifolds with fast curvature decay and maximal volume growth. } Invent. Math. 97 (1989), no. 2, 313-349.
 \bibitem{Bl} Z. Blocki. \emph{ Interior regularity of the complex Monge-Amp\`ere equation in convex domains,} Duke Math. J. 105 (2000), no. 1, 167-181.
\bibitem{Ch1} J. Cheeger. \emph{Degeneration of Riemannian metrics under Ricci curvature bounds} Lezione Fermiane, Scuola Normale Superiore, Pisa 2001.
\bibitem{Ch2} J. Cheeger. \emph{Degeneration of Einstein metrics and metrics with special holonomy}. Surveys in Differential Geometry VIII, 2003, International Press 29-74.
\bibitem{CW} X. Chen and B. Weber. \emph{Moduli spaces of critical Riemannian metrics with $L^{n/2}$ norm curvature bounds}. Adv. Math. 226 (2011), no. 2, 1307-1330.
\bibitem{CN} T. Colding, A. Naber. \emph{Sharp H\"older continuity of tangent cones for spaces with a lower Ricci curvature bound and applications}, arxiv:/1102.5003.
\bibitem{Croke} C. Croke.  \emph{Some isoperimetric inequalities and eigenvalue estimates}, Ann. Sci. \'Ecole Norm Sup.(4)13 (1980), 419-435. 
\bibitem{DT} W-Y. Ding and G. Tian. \emph {K\"ahler-Einstein metrics and the generalised Futaki invariant}. Inventiones Math. 110 (1992) 315-335.
\bibitem{D1} S. Donaldson. \emph{Stability, birational transformations and the K\"ahler-Einstein problem}. arXiv:/1007.4220. To appear in Surveys in Differntial Geometry, Vol XVII, International Press 2012
\bibitem{D2} S. Donaldson. \emph{b-Stability and blow-ups}.  arXiv:/1107.1699. To appear in Proc. Edinburgh Math. Soc.
 \bibitem{EGZ} P. Eyssidieux, V. Guedj, A. Zeriahi. \emph{Singular K\"ahler-Einstein metrics.} J. Amer. Math. Soc. 22 (2009), no. 3, 607-639.
\bibitem{GH} P. Griffiths and J. Harris. \emph{Principles of Algebraic Geometry}. John Wiley, 1978.
\bibitem{H} R. Hartshorne. \emph{Algebraic Geometry}, Springer, 1977
\bibitem{CL} C. Li.  \emph{Ph D thesis}, Princeton 2012
\bibitem{NW} A. Nijenhuis, W. Woolf. \emph{Some integration problems in almost-complex and complex manifolds. } Ann. of Math. (2) 77 1963, 424-489. 
\bibitem{RT} J. Ross and R. Thomas. \emph{Weighted projective embeddings, stability of orbifolds, and constant scalar curvature K\"ahler metrics}. J. Differential Geom. 88 (2011), no. 1, 109-159
\bibitem{Sh} B. Shiffman. \emph{ On the removal of singularities of analytic sets. }Michigan Math. J. 15(1968), 111-120.
\bibitem{Sib}  L. Sibner. \emph{The isolated point singularity problem for the coupled Yang-Mills equations in higher dimensions. }Math. Ann. 271 (1985), no. 1, 125-131. 
\bibitem{Sz} G. Szekelyhidi.  \emph{Filtrations and test-configurations}.  arXiv:1111.4986 
\bibitem{Tian1} G. Tian. \emph{On Calabi's conjecture for complex surfaces with positive first Chern class.} Invent. Math. 101 (1990), no. 1, 101-172.
\bibitem{Tian2} G. Tian \emph{K\"ahler-Einstein metrics on algebraic manifolds}. Proc. of Int Congress Math.  Kyoto, 1990
\bibitem{Tian3}G. Tian. \emph{Existence of Einstein metrics on Fano manifolds}. Preprint 2010
\bibitem{Wehrheim} K. Wehrheim. \emph {Uhlenbeck compactness}. European Math. Soc. (Zurich) 2004
\end{thebibliography}
\end{document}